\documentclass[notitlepage,reqno,12pt]{amsart}
\usepackage{latexsym,amssymb, epsfig, amsmath,amsfonts, subfigure,amsthm, bm}

\usepackage{rotating}
\usepackage[toc,page]{appendix}
\usepackage{color}
\usepackage{multirow}
\usepackage{relsize}
\usepackage{microtype}

\usepackage[foot]{amsaddr}

\usepackage{epic,graphpap,graphics,float,tabularx}
\usepackage{url, epsfig, graphicx, color, times}
\usepackage{pgf,tikz}
\usetikzlibrary{arrows}

\usepackage[colorlinks, allcolors=blue]{hyperref}

\usepackage[margin=1in]{geometry}

\numberwithin{equation}{section}
 
\newtheorem{lemma}{Lemma}[section]

\newtheorem{definition}{Definition}[section]

\newtheorem{proposition}{Proposition}[section]
\newtheorem{remark}{Remark}[section]

\newlength{\defbaselineskip}
\newcommand{\setlinespacing}[1]%
           {\setlength{\baselineskip}{#1 \defbaselineskip}}
\newcommand{\doublespacing}{\setlength{\baselineskip}%
                           {1.5 \defbaselineskip}}

\newcommand{\CC}{{\mathbb C}}

\newcommand{\RR}{{\mathbb R}}
\newcommand{\ZZ}{{\mathbb Z}}
\newcommand{\NN}{{\mathbb N}}

\def\E{\mathbb{E}}

\newcommand{\deq}{\stackrel{\rm d}{=}}
\newcommand{\beql}[1]{\begin{equation}\label{#1}}
\newcommand{\eeq}{\end{equation}}

\newcommand{\beqal}[1]{\begin{eqnarray}\label{#1}}
\newcommand{\eeqa}{\end{eqnarray}}
\newcommand{\beq}{\begin{displaymath}}
\newcommand{\eeqno}{\end{displaymath}}
\newcommand{\bali}[1]{\begin{align}\label{#1}}
\newcommand{\eali}{\begin{align}}
\newcommand{\balino}{\begin{align*}}
\newcommand{\ealino}{\begin{align*}}

\newcommand{\ep}{\epsilon}

\newcommand{\Cov}{\text{\rm Cov}}

\newcommand{\non}{\nonumber}

\newcommand{\baa}{\begin{eqnarray*}}
\newcommand{\eaa}{\end{eqnarray*}}

\newcommand{\ttl}{\Large Semimartingale properties of a generalized fractional Brownian \\[5pt] motion  and its mixtures with applications in asset pricing}

\begin{document}

\title[]{\ttl}

\author{Tomoyuki \ Ichiba$^*$}
\address{$^*$Department of Statistics \& Applied Probability, University of California, Santa Barbara, CA 93106}
\email{ichiba@pstat.ucsb.edu}

\author{Guodong Pang$^{**}$}
\address{$^{**}$Department of Computational Applied Mathematics and Operations Research,
George R. Brown College of Engineering,
Rice University,
Houston, TX 77005}
\email{gdpang@rice.edu}
\author{Murad S. \ Taqqu$^\dag$}
\address{$^\dag$Department of Mathematics and Statistics, Boston University, Boston, MA 02215}
\email{murad@math.bu.edu}

%
%

\begin{abstract} 

We study the semimartingale properties for the generalized fractional Brownian motion (GFBM) introduced by Pang and Taqqu (2019)  and discuss the applications of the GFBM and its mixtures to financial asset pricing.  
 The GFBM is self-similar and has non-stationary increments, whose Hurst index $H \in (0,1)$ is determined by two parameters. We identify the regions of these two parameter values where the GFBM is a semimartingale.  
 
 


We next study the mixed process made up of an independent BM and a GFBM and identify the range of parameters for it to be a semimartingale, which leads to $H \in (1/2,1)$ for the GFBM. 
 We also derive the associated equivalent Brownian measure.
 This result is in great contrast with the mixed FBM with $H \in \{1/2\}\cup(3/4,1]$ proved by Cheridito (2001) and shows the significance of the additional parameter introduced in the GFBM. 
 
 We then study the semimartingale asset pricing theory with the mixed GFBM, in presence of long range dependence, and applications in option pricing and portfolio optimization. Finally we discuss the implications of using GFBM on arbitrage theory, in particular, providing an example of semimartingale asset pricing model of long range dependence without arbitrage.

\end{abstract}

\keywords{Generalized fractional Brownian motion, semimartingale, 
mixture of Brownian motion and GFBM, stock price model, option pricing, portfolio optimization, arbitrage}

\maketitle

\allowdisplaybreaks

\doublespacing

\allowdisplaybreaks

\section{Introduction}

\subsection{Semimartingale properties of GFBM and its mixture}
Semimartingale and non-semimartingale properties of the standard fractional Brownian motion  (FBM) $B^H$ and its mixtures are well understood. These properties are important in modeling stock price \cite{kluppelberg2004fractional,rostek2009option}, constructing arbitrage strategies and hedging policies \cite{rogers1997arbitrage,shiryaev1998arbitrage,salopek1998tolerance,cheridito2003arbitrage}, 
and modeling rough volatility \cite{gatheral2018volatility,bayer2016pricing,roughvolatility-webpage}. 
The standard FBM $B^H$ captures short/long-range dependence, and possesses the self-similar and stationary increment properties, as well as regular path properties. 
It may arise as the limit process of scaled random walks with long-range dependence or an integrated shot noise process \cite{PT17}.

A generalized fractional Brownian motion (GFBM) $X$, introduced by Pang and Taqqu \cite{pang-taqqu}, is defined via the following (time-domain) integral representation: 
\begin{equation} \label{def-X}
\{X(t)\}_{t\in \RR} \deq \left\{  c \int_\RR \left( (t-u)_+^{\alpha} - (-u)_+^{\alpha} \right)    |u|^{-\gamma/2} B({\rm d}u)  \right\}_{t\in \RR},
\end{equation}
where $(\alpha, \gamma)$ are constants in the region 
\begin{equation} \label{def-X-g-c} 
 \quad  \gamma \in [0,1), \quad
\alpha \in \Big(-\frac{1}{2}+ \frac{\gamma}{2}, \ \frac{1}{2}+ \frac{\gamma}{2}\Big)\, ,
\end{equation}
 $B({\rm d} u)$ is a Gaussian random measure on $\RR$ with the Lebesgue control measure ${\rm d}u$, and $c=c(\alpha,\gamma) \in \RR_+$ is a normalization constant (see \eqref{eq: normalizeC}). It is a $H$-self-similar Gaussian process, that is, $\{X(\kappa t): t\in \RR\} \deq \{\kappa^{H}X(t): t \in \RR\}$ for any $\kappa > 0$,   with 
 $$
H:=\alpha - \frac{\gamma}{2} + \frac{1}{2} \in (0,1)\,,$$
but does not have stationary increments. 
The Hurst parameter $H$ is determined by two-parameters $(\alpha,\gamma)$ in the range shown in Figure \ref{fig-1}.  One may regard $\gamma \in (0,1)$ as a scale/shift parameter.  
The multiplier $ |u|^{-\gamma/2}$ in the stochastic integral in \eqref{def-X} renders a higher ( resp., smaller) value of the integrand when $\,u\,$  is close to $\,0\,$ ( resp., goes to $\,\infty\,$).  
On the one hand,  it is somewhat surprising that the GFBM $X$ has the H{\"o}lder continuity property with the same parameter $H-\ep$ for $\ep>0$ as the FBM $B^H$, and that the parameter $\gamma$ does affect the differentiability of the paths  \cite{ichiba-pang-taqqu}.

The GFBM $X$ defined in \eqref{def-X} can be also written as 
\begin{equation} \label{def-X-alternative}
\{X(t)\}_{t\in \RR} \deq \left\{  c \int_\RR \left( (t-u)_+^{H- \frac{1}{2} +\frac{ \gamma}{2} } - (-u)_+^{H- \frac{1}{2} +\frac{ \gamma}{2} } \right)    |u|^{-\gamma/2} B( {\mathrm d} u)  \right\}_{t\in \RR}.
\end{equation}
It is clear that when $\gamma=0$, this becomes the standard FBM $B^H$ with Hurst parameter $H = \alpha +1/2 \in (0,1)$ (equivalently, $\alpha \in ( -1/2,  1/2)$): 
\begin{equation} \label{FBM-rep}
\{B^H(t)\}_{t\in \RR} \deq \left\{  c \int_\RR \left( (t-u)_+^{H-\frac{1}{2}} - (-u)_+^{H-\frac{1}{2}} \right)   B( {\mathrm d} u)  \right\}_{t\in \RR}. 
\end{equation}

In the special case when $H=1/2$, $B^H$ becomes the standard Brownian motion $B(t)$. For the GFBM, when $\gamma \in [0,1)$ and $\alpha=0$, that is, $H\in (0,1/2]$, it becomes 
\begin{equation} \label{eqn-X-alpha0} 
\{X(t)\}_{t\in \RR_+} \deq \left\{  c \int_0^t  u^{-\gamma/2} B({\rm d}u)  \right\}_{t\in \RR_+}  \deq \left\{ B(t^{1-\gamma})  \right\}_{t\in \RR_+}
\end{equation}
with $c= (1-\gamma)^{1/2}$, that is, the GFBM $X$  in \eqref{def-X} reduces to a time-changed Brownian motion with a power-law time change function $t^{1-\gamma}$.  Furthermore, when $\gamma=0$ and $\alpha=0$,  $X$ in \eqref{def-X} becomes the standard Brownian motion $B$ with $H = 1/2$.  

The GFBM $X$ is derived as the limit of integrated power-law shot noise processes in \cite{pang-taqqu} (see a brief review in Section \ref{sec-SNP-limit}).  
We have studied in \cite{ichiba-pang-taqqu} some important path properties of the GFBM $X$, including the H{\"o}lder continuity property, the differentiability and non-differentiability properties, and functional/local law of the iterated logarithm (LIL), see Section \ref{sec-GFBM-review} for a summary of its fundamental properties. Some additional path properties such as the exact uniform modulus of continuity and Chung's LIL are recently studied in \cite{wang2021exact,wang2021lower}.

\begin{figure} 
\begin{center}
\begin{tabular}{cc}
\begin{tikzpicture}[scale=2]
\draw[thick,->] (1, 0) -- (1.3, 0.0) node[anchor=north west] {$\gamma$}; 
\draw[dashed] (0, 0) -- (1, 0); 
\draw[thick,->] (0,-0.8) -- (0, 1.3) node[anchor=south east] {$\alpha$}; 
\draw[] (0,0) node[anchor = north east] {$0$};
\draw[] (0,-0.5) -- (1, 0) node[below] {$1$};
\draw[] (0, 0.5) -- (1,1) node[anchor = north west]{};
\draw[] (0, 0.5) node[anchor=south east] {$1/2$};
\draw[] (0, -0.5) node[anchor=north east] {$-1/2$};
\draw[] (1,0) -- (1,1) node[]{};
\draw[dotted] (0, 1/2) -- (1,1/2) node[]{};
\draw[] (0.5, 0.5) node[anchor=south west]{(I)};
\draw[] (0.5, 0.4) node[anchor=north]{(II)};
\draw[fill=gray, opacity =0.2] ( (0,0.5) -- (1, 0.5) -- (1, 1) -- cycle;
\end{tikzpicture}
&
\begin{tikzpicture}[scale=2]
\draw[thick,->] (1, 0) -- (1.3, 0.0) node[anchor=north west] {$\gamma$}; 
\draw[dashed] (0, 0) -- (1, 0); 
\draw[thick,->] (0,-0.8) -- (0, 1.3) node[anchor=south east] {$\alpha$}; 
\draw[] (0,0) node[anchor = north east] {$0$};
\draw[] (0,-0.5) -- (1, 0) node[below] {$1$};
\draw[] (0, 0.5) -- (1,1) node[anchor = north west]{};
\draw[] (0, 0.5) node[anchor=south east] {$1/2$};
\draw[] (0, -0.5) node[anchor=north east] {$-1/2$};
\draw[] (1,0) -- (1,1) node[]{};
\draw[dotted] (0, 1/2) -- (1,1/2) node[]{};
\draw[] (0.5, 0.5) node[anchor=south west]{(I)};
\draw[] (0.35, 0.45) node[anchor=north]{(II)-1};
\draw[] (0.5, 0.02) node[anchor=north]{(II)-2};
\draw[fill=gray, opacity =0.2] ( (0,0.5) --  (1, 1) -- (1,0.5) -- cycle;
\draw[fill=gray, opacity =0.6] ( (0,0.5) --  (1, 0.5) -- (0, 0) -- cycle;
\end{tikzpicture}
\\
(a) & (b) \\
\end{tabular}
\caption{(a) The GFBM $\,X\,$ in \eqref{def-X} is semimartingale in the region (I): $\,0 < \gamma < 1\,$, $\,1/2 < \alpha < (1+\gamma)/2\,$  with $H \in (1/2,1)$ and in the line segment: $\alpha=0$, $\gamma \in [0,1)$ with $H\in (0,1/2]$; it is 
 not semimartingale in the region (II): $\,0 < \gamma < 1 \,$, $\, \alpha \in ( (\gamma-1)/2,  1/2] \setminus \{0\}\,$ with $H\in (0,1/2)\cup(1/2,1)$. 
 The standard fractional Brownian motion corresponds to the line segment: $\, \gamma = 0 \,$, $\, -1/2 < \alpha < 1/2 \,$ with $\, H = \alpha + 1/2\,$. The standard Brownian motion corresponds to the point $\, \alpha = \gamma = 0\,$. 
(b) The mixed GFBM $\,Y\,$ in \eqref{eq: Y} is semimartingale in the region (I) and the region (II)-1: $\,0 < \gamma < 1 \,$, $\, \gamma/2 < \alpha < (1+\gamma)/2 \,$  with $H \in (1/2,1)$, and the line segment: $\alpha=0$, $\gamma \in [0,1)$  with $H\in (0,1/2]$. 
}
\label{fig-1}
\end{center}
\end{figure}

In this paper we focus on the semimartingale properties associated with the GFBM and its mixture with an independent BM. 
For FBM $B^H$, it is shown in \cite{rogers1997arbitrage} that $B^H$ is not a semimartingale for $H \in (0,1/2) \cup (1/2,1)$. When $H=1/2$, $B^H$ is a Brownian motion, and is a martingale (thus, also a semimartingale). 
In \cite{cheridito2001mixed}, it is shown that the mixture of independent BM and FBM is a semimartingale for $H \in (3/4,1)$. This important finding is proved using a filtering approach in \cite{cai2016mixed}.  
These results have significant implications in financial applications, in particular, arbitrage theory and pricing, see, e.g., \cite{rogers1997arbitrage,cheridito2003arbitrage}. 
The proofs for these results rely heavily upon the stationary increments property of the FBM.  The lack of stationary increments of the GFBM $X$ requires new approaches to establish similar results.

We identify the regions of the two parameter values $(\alpha, \gamma)$ in which the GFBM $X$ is a semimartingale (see 
Figure \ref{fig-1}(a) and Proposition \ref{prop-sm-X}). To do so, we first establish the necessary and sufficient condition for the square integrability of the derivative of the kernel of this Gaussian process.
We use this to distinguish the parameter regions for the semimartingale property. We can then apply the characterization of the spectral representation of Gaussian semimartingales by Basse \cite{basse2009spectral}. 


Recall that the standard FBM $B^H$ is a semimartingale if and only if $H=1/2$ while the GFBM $X$ can be a semimartingale for $H\in (1/2,1)$ in { the triangular} region (I) of parameters of $(\alpha,\gamma)$, 
{and for $H\in (0,1/2]$ in the line segment with $\alpha=0$ and $\gamma \in [0,1)$.
Note that  there is a 
quadrilateral shape of $(\alpha,\gamma)$ that results in $H \in (1/2,1)$, and there is a line segment of $(\alpha,\gamma)$ satisfying $\alpha=\gamma/2$ over $[0,1]$ resulting in $H=1/2$ but only $\gamma=0$ gives a standard BM. 
On the other hand, the line segment $\alpha=0$ and $\gamma \in (0,1)$ resulting in $H\in (0,1/2)$ gives a time-changed BM. 
} 
However,  although $X$ is the semimartingale in region (I), it is a process of finite variation and has an explicit expression, thanks to the result of spectral representation of Gaussian semimartingales in Basse \cite{basse2009spectral}.

We then study the semimartingale properties of the mixed GFBM process (sum in \eqref{eq: Y} of an independent BM and GFBM). 
It is shown in \cite{cheridito2001mixed,cai2016mixed,bender2011fractional} that the mixed  FBM $B^H$ process is a semimartingale with respect to its own filtration if and only if $H\in \{1/2\} \cup (3/4,1]$.
We show in Proposition \ref{prop-mixture} that the mixed GFBM process is a semimartingale in the region of the two parameter values $(\alpha, \gamma)$ that is equivalent to $H\in (1/2,1)$ (see regions (I) and (II)-1 in Figure \ref{fig-1}(b)), 
{as well as in the line segment $\alpha=0$ and $\gamma \in [0,1)$ resulting in $H\in (0,1/2]$.  That is, there exist parameter values $(\alpha, \gamma)$ resulting $H\in (0,1)$ that can make the mixed GFBM a semimartingale. }
\emph{It is worth highlighting  the wide range of values of the Hurst index $H$  for which the mixed GFBM process is a semimartingale when $\gamma > 0$.}  This is one significant consequence of introducing the parameter $\gamma$ in the GFBM. 

In the two regions  (I) and (II)-1, the mixed GFBM process has different behaviors due to the fact that in region (I) the GFBM $X$ is a process of finite variation. 
For both regions (I) and (II)-1,
we can use the characterization of the equivalence of Gaussian measures in Shepp \cite{shepp1966radon}, and show that the absolute continuity of the measure of the mixed GFBM with respect to that of the standard BM, and provide an expression of the Radon--Nikodym density, using the solutions to the associated Wiener--Hopf integral equations. 
For that purpose, we establish that $H \in (1/2,1)$ is the necessary and sufficient condition for the second partial derivative function for the covariance function of the GFBM $X$ to be square integrable (see Lemma \ref{lm:sq-integrability}).
For region (I), thanks to the finite variation property of the GFBM $X$, 
the mixed GBFM becomes a Brownian motion plus  a random drift of finite variation, and as a consequence, we provide another expression of the Radon--Nikodym density, in terms of the conditional expectation, applying the results in \cite{MR281279}. 
 We also conjecture that the mixed process is not a semimartingale when $H\in (0,1/2)$ {except the parameters $(\alpha,\gamma)$ taking values in the line segment $\alpha=0$ and $\gamma \in (0,1)$} (see further discussions in Remark \ref{rem-mixture}).


\subsection{Asset pricing with GFBM and its mixture}

We next discuss the GFBM and its mixtures in the study of  financial asset models, in the aspects of modeling long range dependence, non-stationary increments, semimartingale asset pricing theory and arbitrage.

The study of the long range dependence (LRD) in stock market prices has a long history. Cont \cite{cont2005long} discussed the financial modeling of stock prices with the notions of self-similarity, scaling, fractional process and LRD.
In the literature, shot noise process and FBMs have been used to model LRD in financial markets (see, e.g., \cite{kluppelberg2004fractional,altmann2008shot,schmidt2017shot,Stute}). 
 It is well known that the FBM $\, B^{H}\,$ has stationary increments that have slowly decaying autocovariance functions, if $\, H > 1/2\,$. See, for example,  \cite{cheridito2001mixed,cheridito2003arbitrage,MR2378138,rogers1997arbitrage} for the modeling with the FBM and its mixed processes in the financial markets, and also \cite{guasoni2021high,guasoni2019trading}  for the high frequency trading and the optimization under the transaction costs. 
Willinger et al. \cite{willinger1999stock} discuss the question of whether the actual stock market prices have LRD from the empirical investigation, where the Hurst parameter is estimated as slightly above $\, 0.5\,$ for different financial time series. The relatively recent empirical studies also indicate that the stock price process has the Hurst index greater than $\, 0.5\,$ (see e.g., \cite{FernandezMartinez14}). 
In addition, it is also observed in empirical studies that some financial time series exhibits non-stationary increments (see, e.g., \cite{bassler2007nonstationary}).

 Since GFBM does not have stationary increments, the usual definition of LRD cannot be applied. Here, using Lamperti transform which results in a stationary process, we introduce a definition of LRD for self-similar processes, Definition \ref{def:LRDSS}. We show that the GFBM has LRD if and only if $\, \alpha > 0 \,$ (Proposition \ref{prop:LRD}) in the sense of Definition \ref{def:LRDSS}. (For the standard FBM with $\gamma=0$, $\alpha>0$ is equivalent to $H>1/2$.) 
For GFBM,  in both regions in (I) and (II)-1, with $H \in (1/2,1)$, the LRD property holds, however, in the region $\gamma \in (0,1)$ and $0 <\alpha <\gamma/2$ (the triangle area above zero in region (II)-2, resulting in $H \in (0,1/2)$), the LRD property also holds, which differs from the standard FBM with $\gamma=0$. We focus on only the regions in (I) and (II)-1 because of the associated semimartingale properties. 
(It is worth noting that the self-similarity does not necessarily imply LRD of the stationary increments in general, see further discussions in  \cite{cont2005long}.) 


Semimartingale property is one of the most fundamental tools to modern asset pricing theory, because of the equivalent martingale measures and many other associated probabilistic techniques \cite{KS1998}. 
Using FBM $B^H$ itself in this theory has a drawback since it is martingale if and only if $H=1/2$, the BM case. The mixed FBM then draws attention since it is shown to be a semimartingale when the Hurst index of the FBM $H \in (3/4,1)$, see \cite{cheridito2001mixed,cai2016mixed}. However, as observed from financial data \cite{cont2005long,FernandezMartinez14}, the Hurst index can be less than 0.75 and even close to 0.5. 
Therefore, the mixed GFBM will become very useful since we have shown it is a semimartingale when the Hurst index of the GFBM $H \in (1/2,1)$. 
We note that although GFBM itself is also a semimartingale in region (I) with Hurst parameter $H \in (1/2,1)$, it is `unfortunately' of finite variation, so we must resort to the mixed GFBM as an asset pricing model. 

To facilitate the use of the mixed GFBM as a semimartingale asset pricing model, we derive the Radon--Nikodym derivative for the equivalent martingale measure (Proposition  
\ref{prop-pricing-RN}) and semimartingale decompositions.
 These representations will turn out to be useful for the price dynamics of various options and portfolio optimization (Sections \ref{sec: 6.2.1} and \ref{sec: 6.2.2}). In Section \ref{sec: 6.2.2}, we consider the portfolio optimization of the mixed GFBM asset with the GFBM Hurst parameter $H \in (1/2,1)$ (regions (I) and (II)-1), under the log utility function. 
 It is interesting to highlight that the optimal portfolio remains constant as the mixed FBM case with $H\in (3/4,1)$, which shows the robustness of the semimartingale asset pricing theory (in the wider range of Hurst parameter $H\in (1/2,1)$ with the GFBM). While the mean and variance of the optimal portfolio value at each time depend on $H$ only, the covariance function of the optimal portfolio value at different times is  dependent explicitly on the two parameters $\alpha$ and $\gamma$.  In our framework, the larger range of the Hurst index thus provides greater flexibility in modeling and for further theoretical analysis through the use of  It{\^o}'s formula and the properties of semimartingales.

Motivated by Cheridito \cite{cheridito2001mixed}, we also discuss the semimartingale asset price models with small $\epsilon$-perturbation of the driving  GFBM perturbed by the independent BM. 
In the case of standard FBM with $H \in (3/4,1)$, as $\epsilon \downarrow 0$, the price model becomes non-semimartingale
under which arbitrage exist. For the mixed GFBM with $H\in (1/2,1)$, as as $\epsilon \downarrow 0$, the price model remains a semimartingale if the parameters $(\alpha,\gamma)$ are in region (I), and becomes non-semimartingale
if they are in region (II)-1, while it is interesting to observe that in both cases, arbitrage exists (for different reasons).

Finally, we discuss the implications on arbitrage in asset pricing for stock price processes using the GFBM and its mixtures. With FBM $B^H$, 
arbitrage in fractional Bachelier and Black-Scholes models has been well studied in \cite{rogers1997arbitrage,shiryaev1998arbitrage,salopek1998tolerance,cheridito2003arbitrage}. 
{If one uses the GFBM as a stock price process, we find that the only non-arbitrage scenario is the case of standard BM and time-changed BM (with both $\alpha=0$ and $\gamma\in(0,1)$ resulting in $H\in (0,1/2]$). }
Although the GFBM is a semimartingale in parameter region (I), since it is a process of finite variation, arbitrage exists as shown in \cite{HPS84}.  
On the other hand, if one uses the mixed GFBM  as a stock price process, in the parameter region (I), the process becomes a semimartingale as a BM with a finite-variation drift, and thus, no arbitrage exists. This price process with the mixed GFBM is of particular interest, since it also exhibits the long range dependence, similar to that with the mixed FBM.  
Rogers \cite{rogers1997arbitrage} concluded that it is possible to construct a process similar to the FBM to model LRD of returns while avoiding arbitrage. Hence, we have provided an example of a price model with these desirable properties. 
See more discussions in Section \ref{sec-arbitrage}. 



\subsection{Organization of the paper}

The paper is organized as follows. In Section \ref{sec-GFBM-review}, we give the precise definition of the GFBM and summarize some of its properties. In Sections \ref{sec-sm} 
we study the semimartingale properties of the GFBM.  In Section \ref{sec-mix}, the semimartingale property of the mixed BM and GFBM process is investigated. We present the applications in financial models in Sections \ref{sec-finance-price}, including the LRD property (Section \ref{sec-LRD}), shot-noise process and its scaling limit as GFBM (Section \ref{sec-price}), option pricing (Section \ref{sec: 6.2.1}), portfolio optimization (Section \ref{sec: 6.2.2}) and arbitrage (Section \ref{sec-arbitrage}).  
The details of some of the proofs of the main results are given in Section \ref{sec-app}. 

\bigskip

\section{A generalized fractional Brownian motion} \label{sec-GFBM-review}

The GFBM process $X$ defined in \eqref{def-X}  has the following properties: 
\begin{itemize}
\item[(i)]  $X(0) = 0$ and $\E[X(t)]=0$ for all $t\ge 0$;
\item[(ii)] $X$ is a Gaussian process and $\E[X(t)^2] =  t^{2H}$ for $t \ge 0$;
\item[(iii)] $X$ has continuous sample paths almost surely;
\item[(iv)] $X$ is self-similar with Hurst parameter $H \in (0,1)$;
\item[(v)] the paths of $X$ are H{\"o}lder continuous with parameter $H-\ep$ for $\ep>0$ 
almost surely;
\item[(vi)] 
{
the paths of $X$ is non-differentiable if $\, \alpha \in (-1/2+\gamma/2, 1/2] \,$  and $\gamma \in (0,1)$, and differentiable if $\, \alpha \in (1/2, 1/2+\gamma/2)\,$ and $\gamma \in (0,1)$ almost surely (see non-differentiable region (II)  and differentiable region (I) in Figure \ref{fig-1}).} 
\end{itemize}

These properties are established in \cite{pang-taqqu,ichiba-pang-taqqu}. See Proposition 5.1 \cite{pang-taqqu} for properties (iii) and (iv), and Theorems 3.1 and 4.1 in \cite{ichiba-pang-taqqu} for (v) and (vi). 
Note that the class of continuously differentiable function is a subset of the class of $\alpha$-H\"older continuous functions with $\,0 < \alpha \le 1\,$. 

Also recall that the normalization constant $c= c(\alpha, \gamma) \in \RR_+$ is given by \begin{equation} \label{eq: normalizeC}
\begin{split}
c(\alpha,\gamma) \, :=\, &  \Big( \int^{1}_{0} ( 1 - v)^{2 \alpha } v^{-\gamma} {\mathrm d} v + \int^{\infty}_{0} [(1+v)^{\alpha} - v^{\alpha}]^{2}v^{-\gamma} du	 v \Big)^{-1/2} \\
 \, =\, & \Big( \text{Beta} ( 1 - \gamma, 2\alpha + 1) \\
& {} + \Big( \frac{\,\Gamma ( 1 - \gamma) \,}{\, \Gamma ( - 2 \alpha ) \,} - \frac{\,2 \Gamma ( 1 + \alpha - \gamma) \,}{\,\Gamma ( - \alpha)\,}\Big) \Gamma ( -1 - 2 \alpha + \gamma) \Big) \Big)^{-1/2} \, , 
\end{split}
\end{equation}
as shown in Lemma 2.1 of \cite{ichiba-pang-taqqu}. 
Here, $\, \Gamma ( a) \, :=\, \int^{\infty}_{0} x^{a-1} e^{-x} {\mathrm d} x \,$, $\,a > 0 \,$ and $\, \text{Beta}(a,b) \, :=\, \int^{1}_{0} x^{a-1} (1-x)^{b-1} {\mathrm d} x \,$, $\,a, b > 0 \,$ are the Gamma and Beta functions, respectively, and $\, \Gamma (-a) := (-a)^{-1} \Gamma(-a+1) \,$ for positive non-integer $\,a\,$. 
{When $\alpha=0$ and $\gamma \in [0,1)$, we have $c(0, \gamma) = (1-\gamma)^{1/2}$. }

Recall that in the case $\gamma = 0$, the FBM $B^H$ has stationary increments.
 Namely, 
 the second moment of its increment:
 $$
 \E\bigl[(B^H(s) - B^H(t))^2\bigr] = c^2 |t-s|^{2H}\, , 
 $$
 and the covariance function
  \begin{equation} \label{eqn-cov-FBM}
  \E\bigl[B^H(s) B^H(t)\bigr] = \frac{1}{2}c^2 \big(t^{2H} + s^{2H} - |t-s|^{2H}\big)\,,
 \end{equation}
 where $c=c(\alpha,0) =c(H-1/2,0)$ in \eqref{eq: normalizeC}. 
 
When $\gamma \in (0,1)$, in comparison with the FBM $B^H$, the process $X$ loses the stationary increment property. In particular,  
the second moment of its increment is
\begin{align} \label{eqn-Phi}
\Phi(s,t)&:=\E\bigl[(X(t) - X(s))^2\bigr]  \non\\
&\, = c^2 \int_\RR \Big(  (t-u)_+^{\alpha} - (s-u)_+^{\alpha}  \Big)^2  |u|^{-\gamma} {\rm d}u \non\\
&\,= c^2  \int_s^t (t-u)^{2\alpha} u^{-\gamma} {\rm d}u 
+ c^2\int_0^s ( (t-u)^{\alpha} - (s-u)^{\alpha} )^2 u^{-\gamma} {\rm d}u \non\\
&\qquad  + { c^2} \int_0^\infty ((t+u)^{\alpha} -(s+u)^\alpha)^2 u^{-\gamma} {\rm d}u\, , 
\end{align}
and  the covariance function  is  
\begin{align} \label{eqn-Psi}
\Psi(s,t) & = \Cov(X(t), X(s)) = \E[X(s)X(t)] \non\\
&= c^2 \int_\RR \Big(  \left( (t-u)_+^{\alpha} - (-u)_+^{\alpha} \right)  \left( (s-u)_+^{\alpha} - (-u)_+^{\alpha} \right) \Big)  |u|^{-\gamma} {\rm d}u \non\\
&= c^2 \int_0^s (t-u)^{\alpha} (s-u)^{\alpha} u^{-\gamma} {\rm d}u  \non\\
& \qquad + c^2 \int_0^{\infty} ((t+u)^{\alpha} - u^{\alpha}) ((s+u)^{\alpha} - u^{\alpha})u^{-\gamma}{\rm d}u\, , 
\end{align}
for $0\le s \le t$. %

\medskip 


We also remark that generalized FBMs are stated in a more general form with the additional terms involving $(t-u)_{-}^\alpha - (-u)_{-}^\alpha$ in the integrands in  \cite[Sections 5.1 and 5.2]{pang-taqqu}. In this paper we focus on the representations of $X$ in \eqref{def-X} since the other forms with additional terms can be treated similarly.

\bigskip

\section{When is the GFBM a semimartingale?} \label{sec-sm}

Any centered Gaussian process $X$ with right-continuous sample paths has a spectral representation in distribution \cite{kuelbs1973representation}, that is, 
$$
X(t) \deq \int_{-\infty}^t  K_t(s) {\rm d} N(s), \quad t\ge 0\, ,
$$
where $N$ is an independently scattered centered Gaussian random measure and 
$(t,s) \to K_t(s)$ is a square-integrable deterministic function. 
Basse \cite[Theorem 4.6]{basse2009spectral} characterizes the spectral representation of Gaussian semimartingales, identifying the family of kernels $K_t(s)$ for which the representation $ \big\{ \int_{-\infty}^t K_t(s) {\rm d} N(s): t\ge 0\big\}$ is a semimartingale with respect to the natural filtration $\{\mathcal{F}^N_t: t \ge 0\}$. 
Specifically for one-dimensional case $N(\cdot) \equiv B(\cdot)$ being the Brownian Gaussian random measure, it says that $\{X(t): t\ge 0\}$ is an $\{\mathcal{F}^N_t: t \ge 0\}$ semimartingale if and only if  for $t\ge 0$, the kernel can be represented as 
\begin{equation} \label{eqn-K-g-Psi}
K_t(s) \;=\; g(s) + \int_0^t \Psi_r(s) \mu({\rm d}r)\,,
\end{equation}
where $g:\mathbb R_{+} \to \mathbb R$ is locally square integrable with respect to the Lebesgue measure, $\mu (\cdot)$ is a Radon measure on $\mathbb R_{+} $ and a measurable mapping $\Psi_{r}(s): (r, s) \to \mathbb R $ is square integrable with respect to the Lebesgue measure, $\int_{-\infty}^{\infty} \lvert \Psi_{r}(s)\rvert^{2} {\mathrm d} s = 1 $, and $\Psi_{r}(s) = 0 $, $r < s$.

In the case of FBM $B^H$,  we have the kernel function
$$
K_t(s) = (t-s)_+^{H-1/2} - (-s)_+^{H-1/2}\, ,
$$
and $N(\cdot) = B(\cdot)$.  
Since it is a $\{\mathcal{F}^B_t: t \ge 0\}$ semimartingale if and only if $H=1/2$, i.e., a Brownian motion,
applying \cite[Theorem 4.6]{basse2009spectral},  we have $g\equiv 0$, and $ \Psi_r(s) \equiv 1$.

For the GFBM $X$, we have the kernel function
\begin{equation} \label{eqn-Kt}
K_t(s) = [( (t-s)_{+})^{\alpha} - ((-s)_{+})^{\alpha} ] \, \lvert s \rvert^{-\gamma/2}\, , \quad s \in \mathbb R , \, \, t \ge 0 , 
\end{equation}
and $N(\cdot) = B(\cdot)$ for $\alpha \neq 0$. When  $\alpha = 0 $, $X$ is the time-changed Brownian motion, and the kernel function becomes $K_{t}(s) = s^{-\gamma/2}$ for $s \ge 0$ and $K_{t}(s) \equiv 0 $   for $s < 0 $. 

{Thus, for $ \gamma \in [0,1)$ and $\alpha \in \big(-\frac{1}{2}+ \frac{\gamma}{2}, \ \frac{1}{2}+ \frac{\gamma}{2}\big) \setminus \{0\} $}, define the function 
$\,\Psi_{t}(\cdot) \,$ by 
\[
\Psi_{t}(s)  \, :=\, C_{t}^{-1} \, \big[ \alpha ( t - s)^{\alpha - 1 } s^{-\gamma/2} \cdot {\bf 1}_{ \{0 < s < t \}} + 
 \alpha ( t - s)^{\alpha - 1 } (-s)^{-\gamma/2} \cdot {\bf 1}_{ \{s < 0 \}} \big]  \, , 
 \]
 where $C_{t}$ is a time-dependent, normalizing constant defined by

 \[
 C_{t} :=   \alpha \,  t^{H} \, ( \text{Beta} (1 - \gamma, 2\alpha - 1) + \text{Beta} ( 1 - \gamma, 1 - 2 \alpha + \gamma)){^{1/2}}  \, , \quad t > 0 \, .  
 \]


 \begin{lemma} \label{lm: 3.1}
 The function $\,\Psi_{t}(\cdot) \,$  is square integrable with respect to the Lebesgue measure if $\, 1/2 < \alpha < (1+\gamma)/2\,$ and $\, \gamma \in (0, 1) \,$ (region (I) in Figure \ref{fig-1}(a)). In this case, $\int_{-\infty}^{\infty} \lvert \Psi_{t}(s)\rvert^{2} {\mathrm d} s = 1$.  The function $\,\Psi_{t}(\cdot) \,$  is not square integrable with respect to the Lebesgue measure if $\, \alpha \in (-1/2+\gamma/2,1/2) \setminus \{0\}\,$ and $\, \gamma \in (0, 1) \,$.  
 \end{lemma}
 
 \begin{proof}
 If $\, 1/2 < \alpha < (1+\gamma)/2\,$ and $\, \gamma \in (0, 1) \,$,  
 it follows from the definition that 
 \[
\alpha^{-2}C_{r}^{2} \int^{\infty}_{-\infty} \lvert \Psi_{r}(s)\rvert^{2} {\mathrm d}s 	= \int^{r}_{0} ( r-s)^{2(\alpha - 1)} s^{-\gamma}{\mathrm d} s + 
\int^{0}_{-\infty} ( r-s)^{2(\alpha - 1)} (-s)^{-\gamma}{\mathrm d} s \,  , 
 \]
 where the first and the second terms are rewritten by the change of variables as 
 \begin{equation} \label{eqn-Psi-intp-1}
 \begin{split}
 \int^{r}_{0} (r-s)^{2(\alpha - 1)} s^{-\gamma} {\mathrm d} s & \, =\, r^{2H} \int^{1}_{0} (1-u)^{2(\alpha - 1)} u^{-\gamma} {\mathrm d} u \\
 & \, =\, r^{2H}\text{Beta} ( 1-\gamma, 2 \alpha - 1)  \,, \\
 \int^{0}_{-\infty} (r-s)^{2(\alpha - 1)} (-s)^{-\gamma} {\mathrm d} s & \, =\, r^{2H} \int^{\infty}_{0} (1+u)^{2(\alpha - 1)} u^{-\gamma} {\mathrm d} u  \\
 & \, =\, r^{2H} \text{Beta} ( 1 - \gamma, 1 + \gamma - 2 \alpha ) \, . 
 \end{split}
 \end{equation}
Thus, in this case, $\int^{\infty}_{-\infty} \lvert \Psi_{t}(s)\rvert^{2} {\mathrm d} s \, =\,  1 $. 

On the other hand, if 
the parameters are $\, \alpha \in (-1/2+\gamma/2,1/2) \setminus \{0\}\,$ and $\,\gamma \in (0, 1)\,$,  
then these integrals over the interval $\,(0, 1)\,$ are not integrable.  Thus we conclude the proof. 
 \end{proof} 

\begin{proposition} \label{prop-sm-X}
The following properties hold:
\begin{itemize}
\item[(i)] If $\, \gamma \in (0, 1) \,$ and $\, 1/2 < \alpha < (1+\gamma)/2 \,$ (region (I) in Figure~\ref{fig-1}(a)), 
the GFBM $X$ in \eqref{def-X} is a semimartingale with respect to $\, \mathcal F^{B}(\cdot)\,$,  more specifically, a process of finite variation: 
\begin{equation} \label{eq: dvofX}
\frac{\,{\mathrm d} X(t)\,}{\,{\mathrm d} t \,} \, =\,
c \int^{t}_{-\infty}\Psi_{t}(s)  {\mathrm d} B(s) \, , \quad t \ge 0 \,, 
\end{equation}
where $\,\Psi_{t}(\cdot) \,$ is defined by 
\begin{equation} \label{eq: Psi_{t}(s)}
\Psi_{t}(s)  \, :=\, \alpha ( t - s)^{\alpha - 1 } \lvert s \rvert^{-\gamma/2} 
 \end{equation}
for $\, s < t \,$, and $\, \Psi_{t}(s)  :=  0 \,$ for $\, s > t \,$.

\item [(ii)] 
{If $\, \gamma \in (0, 1) \,$ and $\, \alpha  \in \big(-\frac{1}{2}+ \frac{\gamma}{2}, \ \frac{1}{2}\big) \setminus \{0\}  \,$,  
the GFBM $X$  in \eqref{def-X} is not a semimartingale with respect to $\, \mathcal F^{B}(\cdot)\,$. } 
In particular, if $\gamma=2 \alpha \in (0,1)$ with $H=1/2$, then it is not a semimartingale with respect to $\, \mathcal F^{B}(\cdot)\,$.

\item[(iii)] {If $\gamma \in [0,1)$ and  $\alpha =0$,  the GFBM $X$, as given in \eqref{eqn-X-alpha0}, is a 
 square-integrable martingale (hence a semimartingale)  with respect to $\, \mathcal F^{B}(\cdot)\,$.} 

\item[(iv)] If $\, \gamma = 	0 \,$ and $\, \alpha \in (-1/2, 0) \cup (0, 1/2) \,$, 
the GFBM $X$ in \eqref{def-X} is reduced to a fractional Brownian motion and is not semimartingale.

\end{itemize}
\end{proposition}

%
%
%
%
%

\begin{remark} \label{rem-sm}
In the range (region (I) of Figure~\ref{fig-1}(a)) of parameters  $(\alpha, \gamma)$: $\, \gamma \in (0, 1) \,$ and $\, 1/2 < \alpha < (1+\gamma)/2 \,$, the Hurst parameter $H = \alpha - \frac{\gamma}{2} + \frac{1}{2} \in \big(\frac{1}{2}, 1 \big)$. 
Note that when $\alpha$ is close to $1/2$, and $\gamma$ is close to $1$, the Hurst parameter $H$ is also close to $1/2$, which differs from the standard FBM case with $\gamma =0$ and $\alpha=0$ resulting in $H=1/2$. 
In the line segment $\gamma \in [0,1)$ and  $\alpha =0$, the Hurst parameter $H \in (0,1/2]$. 

Therefore, the range of values of Hurst parameter $H$ possessing the semimartingale property is expanded from a single value $1/2$ for the standard FBM,  { to the entire interval $(0,1/2)$} 
for the process $X(t)$.
{It is important to highlight that the process $X(t)$ is a process of finite variation in region (I) with $H\in (1/2,1)$, and in the line segment $\gamma \in (0,1)$ and  $\alpha =0$, $X(t)$ becomes a time-changed Brownian motion (as given in \eqref{eqn-X-alpha0}), with self-similar parameter  $H \in (0,1/2)$.} 

Moreover, note that $H=1/2$ here only corresponds to the singular point $\alpha=0, \gamma=0$. The GFBM $X$ can have $H=1/2$ on the line segment $\alpha =\gamma/2$, which is a BM if and only if $\alpha=0$. 
\end{remark}

\begin{remark}[Differentiability]  It is shown in \cite{ichiba-pang-taqqu} that the regions (I) and (II) of Figure~\ref{fig-1}(a) correspond to the regions of almost sure differentiable and non-differentiable paths, respectively,   that is, in the region (I) the sample path of GFBM is differentiable, while in region (II) the sample path of GFBM is not differentiable. This is not just a coincidence but it turns out that when $\,\alpha > 1/2\,$, it is a semimartingale and its (local) martingale part in the semimartingale decomposition is zero, and its finite variation part is the integral of a Gaussian process. 
When $\,\alpha \, =\,  \gamma =0\,$, it is a {standard} Brownian motion with non-differentiable sample path. 
{When  $\gamma \in (0,1)$ and  $\alpha =0$, $X$ is a time-changed Brownian motion, also with non-differentiable sample path. } 
\end{remark}

\begin{proof}

(i) Let us consider the case $\, 1/2 < \alpha < (1+\gamma)/2\,$ and $\, \gamma \in (0, 1) \,$. Thanks to the integrability of \eqref{eqn-Psi-intp-1}
in this parameter set, the square integrability of $\, \Psi_{t}(\cdot)\,$ is assured, by Lemma \ref{lm: 3.1}, and hence, by Theorem 4.6 of Basse (2009), we have the representation 
\begin{equation} \label{eqn-X-regionI}
X(t) \, =\,  c \int^{t}_{0} \Big( \int^{r}_{-\infty} \Psi_{r}(s) {\mathrm d} B(s) \Big)  {\mathrm d} r  \,, \quad t \ge 0 \, , 
\end{equation}
where  $\,\Psi_{t}(\cdot) \,$ is defined in \eqref{eq: Psi_{t}(s)}. 
Thus it is a process of finite variation with the first derivative \eqref{eq: dvofX}, in particular, it is a semimartingale. See also section 4 of \cite{ichiba-pang-taqqu}. 

\smallskip


(ii) We show the claim by contradiction. 
Suppose that $\, X(\cdot) \,$ in \eqref{def-X} is a semimartingale with respect to $\, \mathcal F^{B}(\cdot)\,$. We know $\, \mathbb E [ X^{2}(t)] \, =\, 
t^{2\alpha - \gamma +1}\,$ for $\, t \ge 0 \,$. 
Then by Theorem 4.6 of \cite{basse2009spectral} 
 with $\,N_{\cdot} \, =\,B(\cdot)  \,$, $\,g(s) \, :=\, 0\,$, $\, C_{t}\, =\,  (-\infty, t ]\,$, 
there is a canonical decomposition  
\begin{equation}
\begin{split}
\, X(t) \, =\, &\int^{t}_{-\infty} K_{t}(s) {\mathrm d} B(s) \, =\,  \int^{t}_{-\infty} [ ( (t-s)_{+})^{\alpha} - ((-s)_{+})^{\alpha} ] \, \lvert u \rvert^{-\gamma/2} {\mathrm d} B(s) 
\end{split}
\end{equation} 
for $\, t \ge 0 \,$, where the integrand $\, K_{t}(s) \,$ has the form: 
\begin{equation} \label{eq: Kts} 
K_{t}(s) \, :=\, g(s) + \int^{t}_{0} \Psi_{r}(s) \mu ({\mathrm d} r  )  \,, \quad 0 \le s \le t \, . 
\end{equation}
Here $\, g(\cdot) \,$ is square integrable with respect to the Lebesgue measure, $\, \mu (\cdot) \,$ is a Radon measure on $\, \mathbb R_{+}\,$, and $\, \Psi_{t}(s) \,$ is a measurable mapping satisfying
\begin{equation} \label{eq: BasseTHM4.6}
\int^{\infty}_{-\infty}  \lvert \Psi_{t}(s) \rvert^{2} {\mathrm d} s \, =\,  1 \, , \quad \text{ and } \quad \Psi_{t}(s) \equiv 0 \,  \, \,   (s> t)  \, . 
\end{equation}

Taking derivatives with respect to $\, t\,$ in (\ref{eq: Kts}), we have 
\[
\frac{\,{\mathrm d} K_{t}(s)\,}{\,{\mathrm d} t \,}  \, =\,  \Psi_{t}(s) \cdot \frac{\, \mu ( {\mathrm d} t ) \,}{\, {\mathrm d} t \,} \, , 
\]
while the integrand $\,K_{t}(s)\,$ in \eqref{eqn-Kt} 
has the derivative with respect to $\,t \,$: 
\[
\frac{\,{\mathrm d} K_{t}(s) \,}{\,{\mathrm d} t \,} \, =\, 
\alpha ( t - s)^{\alpha - 1 } \lvert s \rvert^{-\gamma/2}\,. 
 \]
 Thus by comparing these two expressions and by setting $\,\mu ( {\mathrm d} r ) \, =\,  {\mathrm d} r \,$
 as the Lebesgue measure, we identify $\, \Psi_{t}(s) \,$ as in \eqref{eq: Psi_{t}(s)}. 
However, as in Lemma \ref{lm: 3.1}, if $\,  \alpha \le 1/2\,$ and $\, \gamma \in (0, 1)\,$, $\, \Psi_{t}(\cdot) \,$ is not square integrable for every $\,t > 0\,$.
This yields a contradiction to (\ref{eq: BasseTHM4.6}). Thus, we claim that $\, X(\cdot) \,$ in \eqref{def-X} is not semimartingale with respect to $\, \mathcal F^{B}(\cdot)\,$, if $\, \gamma \in (0, 1)\,$ and $\, \alpha \le 1/2\,$. The second statement on the parameter sets $\, \gamma \, =\,  2 \alpha \in (0, 1) \,$ with $\, H \, =\,  1/ 2\,$ is proved as a special case. 

\smallskip

(iii) The claim follows from the expression of $X(t)$ in  \eqref{eqn-X-alpha0}. 
\smallskip

 (iv) When $\, \gamma \, =\,  0 \,$, $\, X(\cdot) \,$ in \eqref{def-X} is a FBM with Hurst index $\, \alpha + 1/2\,$, and it is not a semimartingale if $\,\alpha \in (-1/2, 0) \cup (0, 1/2)\,$.
\end{proof}


\bigskip

\medskip

\subsection{Alternative proof of the non-semimartingale property by quadratic variations of the GFBM}
\label{sec-var-sm}

For the standard FBM $B^H$, it is shown in \cite[Proposition 3.14]{MR1785063} that
$B^H$ has a $1/H$-variation, that is, 
$$
\lim_{\varepsilon \downarrow 0} \int_0^t \frac{\,1\,}{\,\varepsilon\,} \lvert B^H(s+\varepsilon) - B^H(s)\rvert^{1/H} {\mathrm d} s \,= \,  \varrho_H t\,
$$
in the sense of convergence uniformly in compact sets, 
where  $\varrho_H = \E[|Z|^{1/H}]$ for $Z\sim N(0,1)$. Here, the limit is in the sense of convergence in probability uniformly on every compact interval (\text{ucp}).  It can be also shown that the classical variation
$$
\sum_{i=0}^{n-1} |B^H_{t_{i+1}} - B^H_{t_i}|^{1/H} \xrightarrow[n \to\infty ]{L^{1}}\, \varrho_H t\,,
$$
where $0=t_0 < \dots < t_n =t$ is a partition of $[0,t]$, see Proposition 3.14 in \cite{MR1785063} and Remark 1 in \cite{MR2218338}.
Then by Propositions 1.9 and 1.11 of \cite{cheridito2001regularizing} one can conclude that $B^H$ is not a semimartingale with respect to $\, \mathcal F^{B}(\cdot)\,$, if $H < 1/2$. Note that the results in \cite{cheridito2001regularizing} involve the notion of the weak semimartingales (see Definition 1.5 of \cite{cheridito2001regularizing}).

This approach of evaluating the variations is also used in \cite{MR2218338} to show that the bifractional Brownian motion $B^{H,K}$ with parameters $\,(H, K)\,$, $\,H \in (0, 1)\,$, $\,K \in (0, 1]\,$ is not a semimartingale, if $HK\neq 1/2$, see Propositions 1--3 and Remark 1 there. Recall that the bifractional Brownian motion $\,B^{H,K}\,$ is a centered Gaussian process with $\,B^{H,K}(0) \, =\,  0 \,$ and covariance function 
\[
\mathbb E [ B^{H,K}(s) \, B^{H,K}(t) ] \, =\,  \frac{\,1\,}{\,2^{K}\,} ( ( t^{2H} + s^{2H})^{K} - \lvert t - s \rvert^{2\, H\, K} ) \,, \quad s, t \ge 0 \, . 
\]
The bifractional Brownian motion $\,B^{H,K}\,$ is a FBM with Hurst index $\,H \in (0, 1) \,$ if $\, K \, =\,  1\,$. 

We next use this variation approach to provide an alternative proof for the non-semimartingale property in 
Proposition \ref{prop-sm-X} (ii) but only restricting to the parameter range $\gamma \in (0,1)$ and   $\alpha \in (-1/2+\gamma/2, 0)$, under which the Hurst parameter $H$ takes values in $(0,1/2)$  (see also Remark \ref{rem-pw}).

\begin{proposition} \label{prop-sm-X-2}
The GFBM $X$ in \eqref{def-X} is not a semimartingale with respect to its own filtration $\, \mathcal F^{X}(\cdot)\,$ 
if 
$\gamma \in (0,1)$ and 
$\alpha \in (-1/2+\gamma/2, 0)$. 
Moreover, 
since $\mathcal F^{X}\subset \mathcal F^{B}$, it is not a
 semimartingale with respect to $\, \mathcal F^{B}(\cdot)\,$ if 
{$\gamma \in (0,1)$ and $\alpha \in (-1/2+\gamma/2, 0)$.
 }
\end{proposition}

\begin{proof}

We write 
\[
X(t) = X_{-}(t) + X_{+}(t),
\]
where
\[
X_{-}(t) = \int_{-\infty}^0 \left( (t-u)^{\alpha} - (-u)^{\alpha} \right)    (-u)^{-\gamma/2} B({\rm d}u)\,,
\]
and
\[
X_{+}(t) = \int_0^t (t-u)^\alpha u^{-\gamma/2}B({\rm d}u)\,. 
\]
Recall that we focus on the parameter range: $\gamma \in (0,1)$ and $\alpha \in (-1/2+\gamma/2, 0)$. In this parameter range, by Lemma 2.1 (valid for all parameter ranges) in \cite{wang2021exact} 
for $0<s <t$,
\[
c_{2,1}\frac{(t-s)^2}{t^{2-2H}} \le \E[(X_{-}(t)  - X_{-}(s) )^2] \le  c_{2,2}\frac{(t-s)^2}{s^{2-2H}}\,,
\]
and by Lemma 3.1 (valid for $\alpha \le 1/2$) in \cite{wang2021exact} , for $s \le 2(t-s)$, 
\[
c_{3,1}\frac{(t-s)^{2\alpha+1}}{t^{\gamma}} \le \E[(X_{+}(t)  - X_{+}(s) )^2] \le  c_{3,2}\frac{(t-s)^{2\alpha+1}}{s^{\gamma}}\,,
\]
and for $t>s >2(t-s)$, 
\[
\E[(X_{+}(t)  - X_{+}(s) )^2] \asymp \frac{(t-s)^{2\alpha+1}}{s^\gamma}\,,
\]
where if $f,g$ are real-valued function on a set $I$, the notation $f\asymp g$ means that $c\le f(x)/g(x)\le c'$ for all $x\in I$ and some positive and finite constants $c,c'$ (which may depend on $f,g$ and $I)$. 

Then we have the lower bound 
\[
\mathbb E[ (X_{+}(u+\varepsilon) - X_{+}(u))^{2}] \ge c \cdot \frac{\,\varepsilon^{2\alpha+1}\,}{\,(u+\varepsilon)^{\gamma}\,} 
\] 
with some constant $\,c\,$, and hence, 
\[
\int^{T}_{0} \mathbb E [ (X_{+}(u+\varepsilon) - X_{+}(u))^{2} ] {\mathrm d} u \ge c \varepsilon^{2\alpha+1} \cdot \frac{\, [ ( T+\varepsilon)^{1-\gamma} - \varepsilon^{1-\gamma}]\,}{\, 1 - \gamma \,} \, . 
\]
This implies that 
\[
\lim_{\varepsilon \downarrow 0 } \frac{\,1\,}{\,\varepsilon\,} \int^{T}_{0} \mathbb E [ (X_{+}(u+\varepsilon) - X_{+}(u))^{2} ] {\mathrm d} u  \, =\, 
\infty, \quad  \text{ if } \alpha  < 0 \, . 
\]
Thus, if $\,\alpha < 0 \,$, the quadratic variation of $\,X\,$ in \eqref{def-X} does not converge and not well-defined. Therefore, it is not semimartingale. 
\end{proof}

\begin{remark} \label{rem-pw}
When $\gamma \in (0,1)$ and $\, \alpha \in (0, 1/2)\,$, which is part of the parameter range of Proposition \ref{prop-sm-X} (ii) for the non-semimartingale property,  it is not clear if this power variation approach in Proposition \ref{prop-sm-X-2} works. Note that in this parameter range, the Hurst parameter $H$ can take value in $(0,1)$, in particular, 
 $H\in (0,1/2)$ if $\alpha \in (0, \gamma/2)$ and $H\in [1/2,1)$ if $\alpha \in [\gamma/2, 1/2)$. 
A careful study of the power variations of the GFBM is worth studying in a future work. 
\end{remark}

\bigskip

\section{Mixed BM and GFBM}  \label{sec-mix}

In this section we consider the semimartingale properties of the following process 
\begin{equation} \label{eq: Y}
Y(t) = \widetilde{B}(t)+ X(t), \quad t\ge 0,
\end{equation}
where $\widetilde{B}(t)$ is a standard Brownian motion and $X(t)$ is the GFBM defined in \eqref{def-X}, independent of $\widetilde{B}(t)$. Let us call $\,Y\,$ the {\em mixed GFBM}.  {In the special case when $\gamma \in [0,1)$ and $\alpha=0$,  $X(t)$ is given in \eqref{eqn-X-alpha0}, and hence
\begin{equation} \label{eqn-Y-alpha0}
Y(t) =  \widetilde{B}(t)+ c \int_0^t  u^{-\gamma/2} B({\rm d}u),
\end{equation}
where $B(t)$ is a standard Brownian motion, independent of $\widetilde{B}(t)$. Recall the corresponding Hurst parameter $H \in (0,1/2]$. In this special case, the mixed process $Y(t)$ is clearly a martingale. In the rest of the section, we focus on the other parameter ranges that make the mixed GFBM a semimartingale. We will show below that in Proposition \ref{prop-mixture}, when $\gamma/2 < \alpha <1/2+ \gamma /2 $ and $\, 0 <\gamma <1 \,$, i.e., regions (I) and (II)-1 in Figure~\ref{fig-1}(b) (both resulting $H\in (1/2,1)$), the mixed GFBM a semimartingale.   In summary,  combining with the case in \eqref{eqn-Y-alpha0}, our results show that the mixed GFBM can be a semimartingale for any value $H \in (0,1)$ (given the parameters $(\alpha, \gamma)$ in the two ranges). 
 }

In the case of FBM $B^H$, we denote
$$
Y^H(t) = \widetilde{B}(t) + B^H(t)\, , \quad H \in (0,1)\, ,  \quad t \ge 0\, .
$$
It is shown in \cite[Theorem~1.7]{cheridito2001mixed} and \cite[Theorem~2.7]{cai2016mixed} (see also \cite{bender2011fractional}) that $Y^H(t)$ is a semimartingale with respect to its own filtration if and only if $H \in \big\{\frac{1}{2}\big\} \cup \big(\frac{3}{4}, 1\big]$.  {On the other hand, the mixed GFBM extends the Hurst parameter values to the full interval $H\in (0,1)$, which further indicates the flexibility by introducing the parameter $\gamma$ in the generalization.}

{To prove the results for the mixed FBM, }
in \cite{cheridito2001mixed}, the concept of weak semimartingale  and a theorem on Gaussian processes in \cite{stricker1984quelques} is used.  On the other hand, 
in \cite{cai2016mixed}, the filtering approach is used. 
In particular, the mixed FBM $Y^H$ is innovated by a martingale in its natural filtration for all $H \in (0,1]$. Then the equivalence property with respect to 
the Wiener measure is established for $H \in (3/4, 1]$ and the equivalence property with respect to  
the Wiener measure is established for $H \in (0,1/4)$. The associated  Radon-Nikodym density formulas are then derived in these ranges of the parameter $H$. 
The method in Section \ref{sec-sm} is not directly applicable for the mixed GFBM in the region (II)-1 (in particular, Lemma 1 only applicable to region (I)). We will use a different approach using Shepp's result for general Gaussian processes  \cite{shepp1966radon} since the Randon-Nikodym derivative of the GFBM exists in the region (II)-1 (see further discussions in Remark \ref{rem-5.3}).

Let $\mu^{Y,H}$ be the probability measure induced by $Y^H$ on the space of its paths in $\CC(\RR_+;\RR)$, and $\mu^B$ be the Wiener measure. 
For $H>1/2$, the covariance function of $B^H(t)$ in \eqref{eqn-cov-FBM} is written as 
\begin{align} \label{eqn-Psi-H}
\Psi^H(t,s) =\E\bigl[B^H(t) B^H(s)\bigr] = \int_0^t \int_0^s K^H(u,v){\rm d}u {\rm d}v,
\end{align}
where 
\begin{align} \label{eqn-K-H}
K^H(t,s) = \frac{\partial^2}{\partial t \partial s}   E\bigl[B^H(s) B^H(t)\bigr]  = c_H |t-s|^{2H-2} . 
\end{align}
with $c_{H}:= c^2 H(2H-1)$. 
If $H> 3/4$,  $K^H(\cdot,\cdot) \in L^2([0,T]^2)$, and $\mu^{Y,H}\sim \mu^B$ (equivalence) by the general criterion in Shepp \cite{shepp1966radon}, and in addition,  Shepp's Radon--Nikodym derivative can be written in the form 
\begin{align*}
\frac{d \mu^{Y,H}}{d \mu^B}(Y^H) = \exp \left( - \int_0^T \varphi_t(Y^H) {\rm d}Y^H(t) - \frac{1}{2} \int_0^T \varphi_t^2(Y^H) {\rm d}t\right) \, , 
\end{align*}
where 
\begin{equation} \label{eqn-varphi-H}
\varphi_{t}(Y^{H}) \, :=\,  \int^{t}_{0} L^{H}(s,t) {\mathrm d} Y^{H}(s) \,, \quad 0 \le t \le T\,  
\end{equation}
and $L^H \in L^2([0,T]^2)$ is the unique solution of the Wiener--Hopf integral equation
\begin{align*}
L^H(s,t) + c_{H} \int_0^t L^H(r,t)|r-s|^{2H-2} {\rm d}r = - c_{H} |s-t|^{2H}, \quad 0 \le s \le t \le T\,. 
\end{align*}

The second partial derivative  $\,K(u, v ) \,$ of the covariance function $\Psi$ in \eqref{eqn-Psi} is given by 
\begin{equation} \label{eq: lm4.1.1}
K(u, v) \, =\,  \frac{\partial^{2} \Psi}{\partial u \partial v}(u,v )  =\,  c^{2}(f_{1}(u\wedge v,u \vee v) + f_{2}(u \wedge v , u \vee v) ) \,,
\end{equation} 
where
\begin{equation*}
\begin{split}
f_{1}(u,v) \, :=&\,  \int^{u}_{0} (v-\theta)^{\alpha -1} (u-\theta)^{\alpha-1} \theta^{-\gamma} {\mathrm d} \theta \,,  \\
f_{2}(u,v) \, :=& \,  \int^{\infty}_{0} (v+\theta)^{\alpha-1}(u+\theta)^{\alpha-1} \theta^{-\gamma} {\mathrm d} \theta \,. 
\end{split}
\end{equation*} 

We then obtain the following square--integrability property of $K(u,v)$. Its proof is given in Section~\ref{sec-app-2}. 

\begin{lemma} \label{lm:sq-integrability} 
Assume $\, 0 < \gamma < 1 \,$ and $\, (\gamma  -1 )/2 < \alpha < (\gamma +1)/2\,$. The function $\,K(\cdot, \cdot) \,$  in \eqref{eq: lm4.1.1} is  square integrable with respect to the Lebesgue measure   in $\,(0, T) \times (0, T) \,$ for every $\, T > 0 \,$, if and only if $\, 0 <\gamma <1 \,$ and $\, \gamma/2 < \alpha <1/2+ \gamma /2\,  $ (equivalently, $H \in (1/2,1)$,  regions (I) and (II)-1 in Figure~\ref{fig-1}(b)). 

%

\end{lemma}

 \begin{remark} \label{rem-5.3}
 We remark that both regions (I) and (II)-1 lead to Hurst parameter $H\in (1/2,1)$. 
The previous lemma shows the integrability of the function $\,K(\cdot, \cdot) \,$  in \eqref{eq: lm4.1.1}, from which we can conclude the absolute continuity of the $Y$ with respect to the  Brownian motion $\, \widetilde{B}(\cdot)\,$ and obtain an expression of the Radon-Nikodym density as a direct consequence of Shepp's result for general Gaussian processes in  \cite{shepp1966radon}.  That will involve the  Wiener--Hopf integral equation. 

 However, the two regions (I) and (II)-1 also have distinct behaviors, despite  the same Hurst parameter range. In particular, in region (I), we have shown in Proposition~\ref{prop-sm-X} that the process $X$ is of finite variation, with the representation in \eqref{eqn-X-regionI}, which makes $Y$ a Brownian motion with a random drift of finite variation.  As a consequence, we are able to provide a more explicit  expression of the Radon-Nikodym density using conditional expectations.
 
We will next state these results in two propositions.  It is an open problem to show that the explicit  Radon-Nikodym density in region (I) is equivalent to the density given by the Wiener--Hopf integral equation. 
 \end{remark}

%

Suppose  that $\gamma/2 < \alpha <1/2+ \gamma /2 $ and $\, 0 <\gamma <1 \,$. Let $\,L(s, t) \in L^{2}([0, T]^{2}) \,$  be the unique solution to the Wiener--Hopf integral equation 
\begin{equation} \label{eqn-WH}
L(s, t) + \int^{t}_{0} L(r, t ) K(r, s ) {\mathrm d} r \, =\,  - K ( s, t) \,, \quad 0 \le s \le t \le T \, , 
\end{equation}
and define 
\begin{equation} \label{eq: phiY}
\, \varphi_{t} (Y) \, :=\, \int^{t}_{0} L(s, t) {\mathrm d} Y (s) \,, \quad 0 \le t \le T \, . 
\end{equation}
Also, let $\, \ell (s, t) \in L^{2}([0, T]^{2}) \,$ be the unique solution to the Volterra equation 
\begin{equation} \label{eqn-WH-V}
\ell (s, t) + \int^{t}_{s} \ell (r, t) L(s, r) {\mathrm d} r \, =\, L(s , t) \,, \quad 0 \le s \le t \le T \, . 
\end{equation}
Thus, by \cite{shepp1966radon} and Lemma \ref{lm:sq-integrability}, we obtain the following proposition.  
\begin{proposition}
\label{prop-mixture}
Suppose that $\gamma/2 < \alpha <1/2+ \gamma /2 $ and $\, 0 <\gamma <1 \,$, i.e., regions (I) and (II)-1 in Figure~\ref{fig-1}(b) (both resulting $H\in (1/2,1)$). 
 The probability measure $\,\mu^{Y}\,$ induced by $\,Y\,$  in \eqref{eq: Y} is absolutely continuous with respect to the Wiener measure $\,\mu^{ \widetilde{B}}\,$ over $\,[0, T ] \,$ with the Radon--Nikodym density 
\begin{equation} \label{eq: RND}
\frac{\,{\mathrm d} \mu^{Y}\,}{\,{\mathrm d} \mu^{ \widetilde{B} }\,} (Y) \, =\, \exp \bigg( - \int^{T}_{0} \varphi_{t} (Y) {\mathrm d} Y (t) - \frac{\,1\,}{\,2\,} \int^{T}_{0} [\varphi_{t}(Y)]^{2} {\mathrm d} t \bigg) \, . 
\end{equation}
By the Girsanov theorem 
\begin{align} \label{eq: Wbar}
\overline{W} (t)   &\, :=\, Y(t) + \int^{t}_{0} \varphi_{s}(Y) {\mathrm d} s  \non\\
& \, =\,  Y({t}) + \int^{t}_{0} \int^{s}_{0} L(r, s) {\mathrm d} Y(r) {\mathrm d} s \,, \quad 0 \le t \le T \, 
\end{align}
is a Brownian motion with respect to its own filtration. Moreover, $\, Y (t) \,$ can be written as 
\begin{equation} \label{eq: semi mart rep Y}
Y(t) \, =\,  \overline{W}(t) - \int^{t}_{0} \int^{s}_{0} \ell (r, s ) {\mathrm d} \overline{W}(r) {\mathrm d} s \,, \quad 0 \le t \le T \, . 
\end{equation} 
Particularly, the filtration $\,\mathcal F^{Y}(\cdot) \,$generated by $\,Y\,$ and the filtration $\,\mathcal F^{ \overline{W}} (\cdot) \,$ satisfy the identities $\, \mathcal F^{ Y} ({t} ) \, =\,  \mathcal F^{ \overline{W}} ({t}) \,$ for $\,  0 \le t \le T \,$. 

Therefore, $Y(t)$ is a semimartingale for the pair $(\alpha, \gamma)$ values in this region. 
\end{proposition}


In region (I) of Figure~\ref{fig-1}(b), since the process $X$ has a finite variation, as expressed in \eqref{eqn-X-regionI}, then the mixed process $Y$ is written as
 \begin{equation} \label{eqn-wt-lambda}
 \begin{split}
 Y(t) \,=\,\widetilde{B}(t) + X(t) \, =& \, \widetilde{B}(t) +   c \int^{t}_{0} \Big( \int^{r}_{-\infty} \Psi_{r}(s) {\mathrm d} B(s) \Big)  {\mathrm d} r\,, \\ = &\, \widetilde{B}(t) + \int^{t}_{0} \widetilde{\lambda}(r) {\mathrm d} r \,, \\
 \text{ with } \quad  \widetilde\lambda(t) \,:= & \, c \int^{t}_{-\infty} \Psi_{t}(s) {\mathrm d} B(s)\,, \quad t \ge 0 \, , 
\end{split} 
\end{equation}
where $ \Psi_{t}(s)$ given in \eqref{eq: Psi_{t}(s)}.
Thus $Y$ is, in fact, a Brownian motion with a random drift of finite variation, particularly, it is a semimartingale. 
By Theorem 2 and Lemma 4 of \cite{MR281279} 
we obtain the  Radon-Nikodym density $
\frac{\,{\mathrm d} \mu^{Y}\,}{\,{\mathrm d} \mu^{ \widetilde{B} }\,} (Y)$ in \eqref{eq: RND}, which is stated in  the following proposition.

\begin{proposition}
\label{prop-mixture-2}
Suppose that $1/2 < \alpha <1/2+ \gamma /2 $ and $\, 0 <\gamma <1 \,$, i.e., region (I) in Figure~\ref{fig-1}(b) (resulting in $H\in (1/2,1)$). 
The Radon--Nikodym density \eqref{eq: RND} over the time interval $\,[0, T]\,$ in Proposition \ref{prop-mixture} is given by 
\begin{equation*}
\frac{\,{\mathrm d} \mu^{Y}\,}{\,{\mathrm d} \mu^{ \widetilde{B} }\,} (Y) \, =\, \exp \bigg( \int_0^T  \mathbb E [ \widetilde\lambda(s) \vert \mathcal F^{Y}(s) ] {\mathrm d} Y(s) - \frac{1}{2} \int_0^T \big(\mathbb E [ \widetilde\lambda(s) \vert \mathcal F^{Y}(s) ] \big)^2 {\mathrm d} s \bigg) \,. 
\end{equation*}
Here,  $\,\varphi_{t}(Y) \,$ in \eqref{eq: phiY} is identified as $\, \varphi_{t}(Y) \, \equiv \, - \mathbb E [ \widetilde\lambda(t) \vert \mathcal F^{Y}(t) ] \,$ for $\, t \ge 0 \,$, and 
\[
\overline{W} (\cdot) \, =\, Y(\cdot) - \int^{\cdot}_{0}  \mathbb E [ \widetilde\lambda(s) \vert \mathcal F^{Y}(s) ]  {\mathrm d} s \, =\,  \widetilde{B}(\cdot) + \int^{\cdot}_{0} \big( \widetilde{\lambda}(s) - \mathbb E [ \widetilde\lambda(s) \vert \mathcal F^{Y}(s) ]  \big) {\mathrm d} s 
\]
is a Brownian motion with respect to its own filtration. 
\end{proposition}


%
%
%

\begin{remark} \label{rem-mixture}
{\rm
We conjecture that the mixture process $Y$ is not a semimartingale with respect to its own filtration in the parameter region {
$\, \alpha \in ( \gamma/2-1/2, \gamma/2] \setminus\{0\}\, $ and $\gamma \in (0,1)$ (region (II)-2 including the boundary line segment $\,\alpha \, =\,  \gamma / 2 \, , \, \gamma \in (0, 1) \,$ in Figure~\ref{fig-1}(b)) and excluding the line segment $\alpha=0$ and $\gamma \in (0,1)$.}
For standard FBM $B^H$, in Cai et al. \cite{cai2016mixed}, representations of the FBM with the Riemann-Liouville fractional integrals and derivatives are used to prove the innovation representations in Theorem 2.4 for $H<1/2$, and equivalence of the measures for $\widetilde{B}+ B^H$ and $B^H$ for $H<1/4$. However, for the GFBM $X$, it still remains open to establish the Riemann-Liouville fractional integrals and derivatives. Therefore, we leave it as future work to prove the non-semimartingale property of the mixture process $Y$ for the parameter pair $(\alpha, \gamma)$ in region (II)-2 of Figure~\ref{fig-1}(b) {excluding  the line segment $\alpha=0$ and $\gamma \in (0,1)$ }. }
\end{remark}

\begin{remark}
{\rm
In this remark, we comment on the generalized Riemann--Liouville (R-L) FBM introduced in Remark 5.1 in \cite{pang-taqqu} and its mixtures. 
It is defined by
\begin{equation} \label{eqn-X-RL}
X(t) = c \int_0^t (t-u)^\alpha u^{-\gamma/2} B({\rm d}u)\, , \quad t \ge 0\, ,
 \end{equation}
 where $B({\rm d}u)$ is a Gaussian random measure on $\RR$ with the Lebesgue control measure ${\rm d}u$ and 
 $c \in \RR$, $\gamma \in [0,1)$ and $\alpha \in (\gamma/2-1/2, \gamma/2+1/2)$. 
 It is a continuous self-similar Gaussian process with Hurst parameter $H=\alpha-\gamma/2+1/2 \in (0,1)$. When $\gamma=0$, it reduces to the standard R-L FBM
 $$
 B^H(t) = c \int_0^t (t-u)^\alpha B({\rm d} u), \quad t\ge 0\, . 
 $$
 {And when $\alpha=0$ and $\gamma\in (0,1)$, it has the expression in \eqref{eqn-X-alpha0}.}
 
 It is clear that the  semimartingale properties in Proposition \ref{prop-sm-X} hold for the process $X$ in \eqref{eqn-X-RL}. 
 In particular,  by letting the natural kernel $K_t(s) := (t-s)^\alpha s^{-\gamma/2}$, we have the spectral representation 
 $
 X(t) \deq \int_0^t K_t(s) N({\rm d}s)\, ,
 $
 for a Gaussian measure $N(\cdot)$. 
 Define 
 $
 \Psi_t(s) = C_t^{-1} \alpha (t-s)^{\alpha-1} s^{-\gamma/2}\, , 
 $
 for a time-dependent normalization constant $C_t$. As shown in the proof of Lemma 3.1, the function $\Psi_1(\cdot)$ is square integrable with respect to the Lebesgue measure if and only if $1/2< \alpha< 1/2 + \gamma/2$ and $\gamma \in (0,1)$, and thus by  Basse's characterization of the spectral representation of Gaussian semimartingales (\cite[Theorem 4.6]{basse2009spectral}), we can conclude the the semimartingale property in part (ii) of Proposition \ref{prop-sm-X}. The non-semimartingale property in part (i) of Proposition \ref{prop-sm-X} also follows from a similar argument as in the proof of the proposition.

 
Similarly, for the mixed  process $Y= \widetilde{B}+X$ with $X$ in \eqref{eqn-X-RL} and an independent BM $\widetilde{B}$,  the properties in Propositions \ref{prop-mixture} and  \ref{prop-mixture-2}  hold. Here one can let
$L(s,t) \in L^2([0,T]^2)$ be the unique solution to the Wiener--Hopf integral equation \eqref{eqn-WH} with
$
K(u,v) =  c^2\int^{u}_{0} (v-\theta)^{\alpha -1} (u-\theta)^{\alpha-1} \theta^{-\gamma} {\mathrm d} \theta\,,
$
  $\ell(s,t)\in L^2([0,T]^2)$ be the unique solution to the Volterra equation \eqref{eqn-WH-V}, and $\widetilde\lambda(t) =  \int^{t}_{0} \Psi_{t}(s) {\mathrm d} B(s)$ in \eqref{eqn-wt-lambda}.

 } 

\end{remark}

\bigskip

\section{Asset pricing  with GFBM and its mixture} \label{sec-finance-price}

\subsection{On the long-range dependence of GFBM} \label{sec-LRD} 
The GFBM $\, X\,$ in \eqref{def-X} is a self-similar process but it does not have the stationary increments, while the FBM $\,B^{H}\,$ is a self-similar process and it has the stationary increments. In order to compare the covariance decays between the FBM and the GFBM, we  introduce the concept of long range dependence (LRD) of self-similar processes below. For it, recall from \cite{MR1781002} \cite{MR770625} that for a self-similar process $\, \xi (t) \,$, $\, t \ge 0 \,$ with the Hurst index $\, H  \ge 0 \,$ and $\, \xi(0) \, =\,  0 \,$, its Lamperti transform $\, \{ \eta (t) , t \in \mathbb R\}\,$ is defined by 
\begin{equation}
\, \eta(t) \, :=\, e^{-H t} \xi (e^{t}) \,, \quad \, t \in \mathbb R \,. 
\end{equation}
The Lamperti transform $\, \eta \,$ of the self-similar Gaussian process with the Hurst index $\,H\,$ is strictly stationary with the covariance function 
\begin{equation}\label{eqn-Ceta}
\Cov(\eta(t+s), \eta(s)) \,=\, \mathbb E [ \eta(t+s) \eta (s) ] \, =: \, C_{\eta}(t) 
\end{equation} 
for $\, s, t \in \mathbb R\,$. Note that $\, \mathbb E [ \xi(t) \xi(s) ] \, =\,  (ts)^{H} C_{\eta}(\log (t/s)) \,$, $\, s, t \ge 0 \,$. 

\begin{definition}[Long range dependence of self-similar processes]\label{def:LRDSS} We say a self-similar process $\,\{ \xi (t) \,$, $\, t \ge 0\} \,$ with the Hurst index $\, H \,$ is long-range dependent, if its Lamperti transform $\, \eta \,$ is long-range dependent in the sense that 
\begin{equation} \label{eq: LRDSS}
\lim_{t\to \infty} \frac{\,1\,}{\,t\,} \log \lvert C_{\eta}(t) \rvert + H  > 0 \,. 
\end{equation}
\end{definition}

This means that the covariance function $\,C_{\eta}(t)\,$ decays slower than $\, e^{ - H t}\,$, as $\, t \,$ goes to infinity. Recall that the Lamperti transform $\, \eta (\cdot) \,$of a standard BM $\, \xi(\cdot) := B(\cdot) \,$ with $\,H = 1/2\,$ is the Ornstein-Uhlenbeck process $\, \eta(t) \, =\, e^{-t/2} B(e^{t/2}) \,$, $\, t \in \mathbb R\,$ with the covariance function $\, C_{\eta}(t) \, =\, e^{-t/2}\,$, $\, t \in \mathbb R\,$. Thus, $\, (1/t) \log ( C_{\eta}(t)) + H \, =\,  0  \,$, $\, t \in \mathbb R\,$, and hence, the standard BM  is not long-range dependent in the sense of Definition \ref{def:LRDSS}. Moreover, the Lamperti transform $\, \eta (t) \, :=\, e^{-Ht} B^{H}(e^{t}) \,$, $\, t \in \mathbb R\,$ of the FBM $\, B^{H}(\cdot)\,$ is strictly stationary and it has the covariance function 
\begin{equation} \label{eq: CovLampertiFBM}
C_{\eta}(t) \, =\,  \cosh (Ht) - 2^{2H-1} ( \sinh ( t/2))^{2H} \,, \quad t \in \mathbb R \, . 
\end{equation}
In this case the limit \eqref{eq: LRDSS} is given by 
\begin{equation} \label{eq: CovLampertiFBM2}
\lim_{t\to \infty} \frac{\,1\,}{\,t\,}  \log (C_{\eta}(t)) + H  \, =\,  2H - 1
\end{equation}
and hence, the FBM is long-range dependent in the sense of Definition \ref{def:LRDSS} if and only if $\, H > 1/2 \,$. This range $\, H  > 1/2\,$ corresponds to the LRD of the stationary increment of the FBM in the usual sense. The proof of \eqref{eq: CovLampertiFBM2} follows from \eqref{eq: 71} in Appendix. The following result (Proposition \ref{prop:LRD}) generalizes this observation. Its proof is given in Section  \ref{app:LRD} in Appendix.

\begin{proposition}[Long-range dependence] \label{prop:LRD} The GFBM $\, X \,$ in \eqref{def-X} is long-range dependent in the sense of Definition \ref{def:LRDSS}, if and only if $\, \alpha > 0 \,$. Particularly, when $\, \gamma \, =\,  0\,$, the FBM $\, B^{H}\,$ is long-range dependent if $\, H = \alpha + 1/2 > 1/2\,$. 
\end{proposition}

\begin{remark}
Definition \ref{def:LRDSS} describes the LRD of the stationary generator $\, \eta (\cdot) \,$ of the self-similar processes $\, \xi (\cdot)\,$. Self-similar processes are not necessarily stationary and the self-similarity does not necessarily imply the LRD of the stationary increments (cf. \cite{cont2005long} \cite{rogers1997arbitrage}). It follows from Proposition \ref{prop:LRD} that if $\, 0 < \gamma < 1 \,$ and $\, 0 < \alpha < \gamma / 2\,$, the GFBM $\, X\,$ in \eqref{def-X} is self-similar process with the Hurst index $\, H \in (0, 1/2)\,$ and is long-range dependent in the sense of Definition \ref{def:LRDSS}. This is different from the understanding of the Hurst index of the FBM (when $\, \gamma \, =\,  0\,$). 
\end{remark}

In the following sections we shall discuss the use of the GFBM with $\, H \in (1/2, 1)\,$ and its mixed processes for modeling the financial assets with long range dependence, by taking advantage of the semimartingale properties. 

\medskip

\subsection{Shot noise process with non-stationary noise and integrated shot noise}  \label{sec-SNP-limit}

In modeling of financial markets, Brownian motion and compound Poisson processes or more generally, L\'evy processes are widely utilized to capture effects of various noises. 
Recently, Pang and Taqqu \cite{pang-taqqu} studied a non-stationary, power-law shot noise process $\mathcal Z^*= \{\mathcal Z^*(y): y \in \RR\}$ on the whole real line defined by
\begin{equation} \label{eq: Zast}
\mathcal Z^*(y) := \sum_{j=-\infty}^\infty g^*(y-\tau_j) R_j\, , \quad y \in \RR\, , 
\end{equation}
where $\{\tau_j: j \in \ZZ\}$ is a sequence of Poisson arrival times of shots with rate $\lambda$ on the whole real line, and each $R_j$ is the noise associated with shot $j$ at time $\tau_j$ for $\, j \in \mathbb Z\, $. It is a generalization of the compound Poisson process on the positive half line. The variables $\{R_j, j \in \ZZ\}$ are conditionally independent, given $\{\tau_j\}$, and the marginal distribution of $R_j$ depends on the shot arrival time $\tau_j$, that is, $ \mathbb P(R_j \le r \, |\,  \tau_j=u) =:  F_u(r)$, $\,r \, \in\,  \mathbb R\,$, $\, u \, \in \,  \mathbb R\,$ for every $\,j \in \mathbb Z\,$. Assume that 
\begin{itemize} \item[(i)] (the power-law property) the function $\,g^{\ast} \,$ satisfies  $g^*(y) =  y^{-(1-\alpha)} L^*(y)$ for $y\ge 0$ and $g^*(y)=0$ for $y <0$ and $\alpha \in (0, 1/2)$, where $L^*$ is a positive slowly varying function at $+\infty$, and 

\item[(ii)] (the moment conditions) the common conditional distribution $\,F_{t} \,$ of the noises $R_\cdot$, given $\,\tau_{\cdot} = t\,$, satisfies the zero mean $K_1(t) := \int_\RR r {\rm d} F_t(r) =0$ for every $\, u \in \mathbb R\,$ and finite variance $K_2(t) = \int_\RR r^2 {\rm d} F_t(r) = t^{-\gamma} \tilde{L}_+(t)$ for $t>0$ and $K_2(t) = \lvert t \rvert ^{-\gamma} \tilde{L}_{-}(t)$ for $t<0$, where $\gamma \in (0,1)$, and $\tilde{L}_{\pm}$ are some positive slowly varying functions. 
\end{itemize}

The shot noise process $\, \mathcal  Z^{\ast}\,$ in \eqref{eq: Zast} is a generalization of the compound Poisson process, because if $\, g^{\ast}(y) \, :=\, {\bf 1}_{\{y > 0 \}} \,$, $\, y \in \mathbb R\,$, then $\, \mathcal  Z^{\ast} \,$ is a compound process. 
The integrated shot noise process $\mathcal Z = \{\mathcal Z(t): t \in \RR_+\}$ is defined by
\begin{equation} \label{eq: Z}
\mathcal Z(t) := \int^{t}_{0} \mathcal Z^{\ast}(y) {\mathrm d} y \, =\,  \sum_{j=-\infty}^\infty (g(t-\tau_j) - g(-\tau_j)) R_j\, , \quad t \ge 0\, ,
\end{equation}
where the shot shape function $\,g(\cdot) \,$ is differentiable with its derivative $\,g^{\ast} \,$, i.e.,  $g(t) := \int_0^t g^*(y){\rm d}y$ for $\,t \ge 0 \,$. 

Pang and Taqqu \cite{pang-taqqu} have shown that the scaled process $\widehat{\mathcal Z}^\varepsilon(t) := \varepsilon ^{H} \mathcal Z(t/\varepsilon) $ converge weakly to the GFBM $X$, as $\,\varepsilon \to 0 \,$.

\medskip 

\subsection{Stock price models driven by the shot-noise with non-stationary noise} \label{sec-price}


 Stock pricing models with a shot-noice component have been developed to study credit and insurance risks \cite{altmann2008shot,schmidt2017shot,Stute}.
In particular, the stock price $P(t)$ is modeled as
\begin{equation} \label{eq: P1}
P(t) := P(0) \exp \left( \Big( \mu  -\frac{\sigma^2}{2} \Big) t + \sigma \widetilde{B}(t)   + \sigma\int_0^t \sum_{\tau_i \le s } \mathfrak f (s-\tau_i, R_i) {\mathrm d} s  \right)\,,\end{equation} 
for $t\ge 0$, 
where $\{(\tau_i, R_i): i \in \NN\}$ is a marked point process, independent of the Brownian motion $\widetilde{B}$,  with arrival times $\tau_i$ and marks (noises) $U_i \in \RR^d$, and the function $\mathfrak f: \RR_+\times \RR^d \to \RR$ is the deterministic shot shape function. $\, \mu \in \mathbb R \,$ and $\,\sigma > 0 \,$ are some real constants. Equivalent martingale measures for this price process are studied in \cite{schmidt2017shot,schmidt2007shot}. In \cite{schmidt2017shot}, it is also discussed when the shot-noise component is Markovian or a semimartingale.  
This is usually when the function $\,\mathfrak f\,$ takes a particular form (exponential function for the Markovian property). 
In these studies, the noises $\{R_i\}$ are assumed to be i.i.d. with finite variance. Since it is usually more difficult to work with the shot noise process directly, one may use the diffusion approximations. For example, Kl{\"u}ppelberg and K{\"u}hn \cite{kluppelberg2004fractional} showed that under regular variation conditions, a Poisson shot noise process can be approximated by an FBM (under proper scaling and validated by a functional central limit theorem), and then used the limiting FBM as a stock pricing model. 

Here, as a pre-limit, we consider the usual random walk noise and the shock noises on the stock price. We evaluate the effects of these noises, when the frequency of arrival of shot noises is very high with appropriate scaling. 
Given a scaling parameter $\, \varepsilon > 0 \,$, we model a pre-limit of price process by 
\begin{align} \label{eq: Pepsilon}
P_{\varepsilon}(t) \, :=\, &  P ( 0) \exp \left( \Big( \mu  -\frac{\sigma^2}{2} \Big) t + \sigma \varepsilon^{1/2} \sum_{j=1}^{ \lfloor t/ \varepsilon \rfloor  } \xi_{j}  + \sigma\int_0^t  \frac{\,1\,}{\,\varepsilon ^{1-H}\,}\mathcal Z^{\ast} \Big( \frac{\, u\,}{\,\varepsilon \,}\Big) {\mathrm d} u \right)  \nonumber \\
\, =\,  & P ( 0) \exp \left( \Big( \mu  -\frac{\sigma^2}{2} \Big) t +\sigma \varepsilon^{1/2} \sum_{j=1}^{ \lfloor t/ \varepsilon \rfloor  } \xi_{j}    + \sigma \varepsilon^{H}\int_0^{t/\varepsilon} \mathcal  Z^{\ast} ( {\, u\,}) {\mathrm d} u \right) \nonumber \\
\, =\, &P ( 0) \exp \left( \Big( \mu  -\frac{\sigma^2}{2} \Big) t + \sigma\varepsilon^{1/2} \sum_{j=1}^{ \lfloor t/ \varepsilon \rfloor  } \xi_{j}   + \sigma \varepsilon^{H} \mathcal Z \Big( \frac{\,t\,}{\,\varepsilon\,} \Big) \right ) \,, 
\end{align}
for $t\ge 0$, 
where $\, \{ \xi_{j}, j \in \mathbb N\}\,$ are i.i.d. random variables with zero mean and unit variance, independent of the shot noise $\,\mathcal Z^{\ast}\,$ in \eqref{eq: Zast}, and  $\, \mathcal Z \,$ is the integrated shot noise process in \eqref{eq: Z}. Here, $\, \mu \,$ and $\,\sigma  > 0 \,$  are some real constants. (One may also choose a model without the random walk component, in which case the model in \eqref{eq: price with shot noise} will have only the process $X$ instead of the mixed GFBM.)



Recall that  $\widehat{\mathcal Z}^\varepsilon := \varepsilon ^{H} \mathcal Z(\cdot/\varepsilon) \Rightarrow X$ and   the random walk term $\, \varepsilon^{1/2} \sum_{j=1}^{ \lfloor t / \varepsilon \rfloor} \xi_{j}\,$, $\, t \ge 0 \,$  converges weakly to the standard BM, independent of $\, X \,$.  
As a scaling limit of \eqref{eq: Pepsilon}, we propose a stock price model using the mixed GFBM as follows: 
\begin{equation} \label{eq: price with shot noise}
\begin{split}
P(t) \, = &\,  P(0)  \exp \left( \Big( \mu  -\frac{\sigma^2}{2} \Big) t + \sigma(\widetilde{B}(t)  + X(t))   \right) \\
\, = & P(0)  \exp \left( \Big( \mu  -\frac{\sigma^2}{2} \Big) t + \sigma Y(t)    \right) \,  
, \quad t\ge 0 \, , 
\end{split}
\end{equation}
where $\, Y (\cdot) \, =\,  \widetilde{B} + X \,$ is the mixed GFBM in \eqref{eq: Y} with $\,X\,$ in \eqref{eqn-X-RL} 
(For a recent account of weak convergence in financial models, we refer to Kreps~\cite{kreps2019black}.)

Suppose that the parameters $\,(\alpha, \gamma)\,$ are in the semimartingale region: $\gamma/2 < \alpha <1/2+ \gamma /2 $ and $\, 0 <\gamma <1 \,$ (i.e., $H\in (1/2,1)$ for the GFBM $\, X\,$), as in the assumption of  Proposition \ref{prop-mixture}. 
Within this parameter range, $\, Y \,$ is a semimartingale. 
{
When the parameters satisfy
$\gamma \in [0,1)$ and $\alpha=0$,  $Y(t)$ is a martingale as given in \eqref{eqn-Y-alpha0}. In that case, the standard asset pricing theory can be applied, so we do not consider it in this section.
}

The price dynamics is determined as the unique strong solution of the linear stochastic differential equation 
\begin{equation} \label{eq: price with shot noise2}
{\mathrm d} P(t) \, =\,  P (t) ( \mu {\mathrm d} t + \sigma {\mathrm d} Y(t) ) \,, \quad t \ge 0 \, , 
\end{equation}
driven by the semimartingale $\, Y\,$, where $\, \mu \,$ is a drift and $\, \sigma \,$ is volatility of stock price under a filtered probability space $\, (\Omega, \mathcal F, (\mathcal F_{t}) , \mathbb P )$. Thus, we take the filtration $\, \mathbb F \, :=\, (\mathcal F_{t}, t \ge 0 ) \, =\, (\mathcal F^{Y}(t), t \ge 0 ) \,$.

%
%
%

\medskip

As important applications with this stock price model \eqref{eq: price with shot noise}, we consider an investor who trades this stock with price \eqref{eq: price with shot noise} and money market account with  an instantaneous interest rate $r (> 0)$ in the following sections \ref{sec: 6.2.1}-\ref{sec: 6.2.2}. 
Propositions \ref{prop-pricing-RN}-\ref{prop: 6} below generalize the previous results in \cite{bender2011fractional},  \cite{cheridito2001mixed} and \cite{MR1912144} from the mixed FBM for the Hurst index $\, H \in (3/4, 1) \,$ to the mixed GFBM for the wider range of the Hurst index with $H \in (1/2, 1)$. 
In particular, for a stock price model with a FBM  $B^H$,  Cheridito  \cite{cheridito2003arbitrage} shows that it admits arbitrage for $H\in (0,1/2) \cup (1/2,1)$ and Chridito  \cite{cheridito2001mixed} shows that for a stock price model with a mixed FBM ($\widetilde{B}+B^H$), the arbitrage can be excluded if $H \in (3/4,1)$. 
 In \cite{MR2447408} and \cite{MR1701227}, no arbitrage condition for smooth strategies is considered under the non-semimartingale stock models (e.g., mixed FBM for $\,H \in (1/2,3/4)\,$) via the path-wise stochastic calculus. Our results show that for a semimartingale stock price model with a mixed GFBM, {(excluding the special case in \eqref{eqn-Y-alpha0})}, 
 all arbitrage strategies can be excluded for a wider range of $H \in (1/2,1)$. That is a significant consequence of the parameter $\gamma$.

\medskip

\subsection{Option pricing with the mixed GFBM}  \label{sec: 6.2.1}

We first recall what is known  in the case of standard FBM $B^H$, as shown in \cite{rogers1997arbitrage,cheridito2001mixed,cai2016mixed}, the mixed process $ Y^H=\widetilde{B}+ B^H$ is a semimartingale if and only if $H =1/2$ (the Brownian case) and  $H \in (3/4,1)$. Of course, with a BM, i.e., $\,H \, =\,  1/2\,$, the standard results of stock pricing and equivalence of martingale measure can be applied. On the other hand, with $H \in (3/4,1)$, one also obtain the Radon--Nikodym derivative in \eqref{eqn-dQdP} where $Y$ is replaced by $Y^H$, and the function $\, \varphi_{t}(Y)\,$ is replaced by $\varphi_{t}(Y^{H})$ in  \eqref{eqn-varphi-H}.  

Also, recall that for the GFBM $X$ with $H=1/2$, the parameters $(\alpha, \gamma)$ lies on the line segment $\alpha=\gamma/2$ in Figure \ref{fig-1}. In this parameter set, the GFBM is a semimartingale, only when $\gamma=0$,  which becomes the special case of Black-Scholes pricing model.
{On the other hand, in the line segment $\gamma \in (0,1)$ and $\alpha=0$, the GFBM is also a semimartingale as given in \eqref{eqn-X-alpha0}, so that Black-Scholes pricing model can also applied with some modification  although it is a time-changed Brownian motion with Hurst parameter $H\in (0,1/2)$.}
Otherwise, there does not exist an equivalent martingale measure. 
Moreover, although GBM $X$ is a semimartingale in region (I), it is of finite variation, so it cannot be used in the arbitrage-free asset pricing framework. Therefore we use consider asset pricing using the mixed GFBM $Y$ in the parameter range $\gamma/2 < \alpha <1/2+ \gamma /2 $ and $\, 0 <\gamma <1 \,$.

In the following proposition, we give the expressions of the  equivalent martingale measures for the discounted stock price process modeled with the mixed process $Y$.  

\begin{proposition}\label{prop-pricing-RN}
 Assume $\gamma/2 < \alpha <1/2+ \gamma /2 $ and $\, 0 <\gamma <1 \,$ (regions (I) and (II)-1 in Figure~\ref{fig-1}(b)). Under the stock price process dynamics \eqref{eq: price with shot noise2}, the discounted stock price process $\, e^{-r t} P(t) \,$, $\, 0 \le t \le T\,$ is a martingale under the new measure $\, \mathbb Q\,$ defined by 
\begin{equation} \label{eqn-dQdP}
\frac{\,{\mathrm d} \mathbb Q\,}{\, {\mathrm d} \mathbb P \,} \bigg \vert_{\mathcal F_{T}}\, :=\,  \exp \bigg( - \int^{T}_{0} (\theta  - \varphi_{t}(Y) ) {\mathrm d} Y(t) - \frac{\,1\,}{\,2\,} \int^{T}_{0} ( \theta^{2} - \lvert \varphi_{t}(Y) \rvert^{2}) {\mathrm d} t \bigg) \, , 
\end{equation}
where $\, \theta \, :=\,  (\mu - r ) / \sigma \,$ is the market price of risk and $\, \varphi_{\cdot}(Y)\,$ is defined in \eqref{eq: phiY}. Particularly, if $1/2 < \alpha <1/2+ \gamma /2 $ and $\, 0 <\gamma <1 \,$, i.e., region (I) in Figure~\ref{fig-1}(b), then $\, \varphi_{\cdot}(Y) \, =\,  - \mathbb E [ \widetilde{\lambda} (\cdot) \, \vert \, \mathcal F^{Y}(\cdot) ] \, =\,  - \mathbb E [ \widetilde{\lambda} (\cdot) \, \vert \, \mathcal F_{\cdot} ]  \,$, 
where 
$ \widetilde\lambda(t)= \int^{t}_{0} \Psi_{t}(s) {\mathrm d} B(s)$ for $t\ge 0$. 

\end{proposition}

\begin{proof}
It follows from Proposition \ref{prop-mixture} and the Girsanov theorem that the process $\, \overline{W} (\cdot) \, =\,  Y(\cdot) + \int^{\cdot}_{0} \varphi_{s}(Y) {\mathrm d} s \,$ in \eqref{eq: Wbar} 
is a Brownian motion for $\,0 \le t \le T\,$ under $\, \mathbb P\,$. 

By the simple application of the product rule to \eqref{eq: price with shot noise2}, we have the discounted stock price process 
\begin{equation*}
\begin{split}
e^{-r t} P (t) \, &=\,  P(0) + \int^{t}_{0} \sigma e^{-r s} P(s) {\mathrm d} \Big( Y(s) + \frac{\,\mu - r \,}{\,\sigma\,} s \Big) \\
& \, =\, P(0) + \int^{t}_{0} \sigma e^{-r s} P(s) {\mathrm d} \Big( \overline{W}(s) - \int^{s}_{0} \varphi_{u}(Y) {\mathrm d}u  + \theta s \Big)   \,,
\end{split}
\end{equation*}
for $t\ge 0$, 
with $\, \theta \, :=\, (\mu - r)  \, / \, \sigma \,$. By another application of the Girsanov theorem, $\, Y(t) + \theta t \, 
=\,  \overline{W}(t) + \int^{t}_{0} (\theta - \varphi_{s}(Y)){\mathrm d} s 
\,$, $\, 0 \le t \le T \,$ is a Brownian motion 
under the new measure $\, \mathbb Q \,$ defined by \eqref{eqn-dQdP}, namely, 
\begin{equation}
\begin{split}
\frac{\,{\mathrm d} \mathbb Q\,}{\,{\mathrm d} \mathbb P \,} \Big \vert_{\mathcal F_{T}} \, :=\, & \exp \Big( - \int^{T}_{0} ( \theta - \varphi_{t}(Y) ) {\mathrm d} \overline{W}(u) - \frac{\,1\,}{\,2\,} \int^{T}_{0} ( \theta - \varphi_{u}(Y))^{2} {\mathrm d} u \Big)\\
\, =\, & \exp \Big( - \int^{T}_{0} (\theta  - \varphi_{t}(Y) ) {\mathrm d} Y(t) - \frac{\,1\,}{\,2\,} \int^{T}_{0} ( \theta^{2} - \lvert \varphi_{t}(Y) \rvert^{2}) {\mathrm d} t \Big) \, .
\end{split}
\end{equation}
In particular, the discounted price process $\, e^{-r t } P(t) \,$, $\, 0 \le t \le T\,$ is a martingale under $\,\mathbb Q\,$.   

If $1/2 < \alpha <1/2+ \gamma /2 $ and $\, 0 <\gamma <1 \,$, i.e., region (I) in Figure~\ref{fig-1}(b), then $\, \varphi_{\cdot}(Y) \, =\,  - \mathbb E [ \widetilde{\lambda} (\cdot) \, \vert \, \mathcal F^{Y}(\cdot) ] \, =\,  - \mathbb E [ \widetilde{\lambda} (\cdot) \, \vert \, \mathcal F_{\cdot} ]  \,$, because  $\, \mathbb F^{Y} \equiv \mathbb F^{P} \, =\, \mathbb F\,$. 
\end{proof}

Consequently, the time-$\,t\,$	 price of European option on this stock with payoff function $\, \mathfrak g \,$ and with maturity $\, T\,$ is given by 
\begin{equation*}
\mathbb E ^{\mathbb Q} [ e^{-r (T-t)}\,  \mathfrak g(P(T)) \, \vert \, \mathcal F_{t}] \, =\,  \mathbb E ^{\mathbb P } \bigg[ e^{-r (T-t)} \, \mathfrak g(P(T))  \cdot \frac{\, {\mathrm d} \mathbb Q\,}{\, {\mathrm d} \mathbb P \,} \Big \vert _{\mathcal F_{T}} \, \, \bigg \vert \mathcal F_{t}  \bigg] \, , 
\end{equation*}
where the conditional expectations $\, \mathbb E^{\mathbb Q}\,$ and $\, \mathbb E^{\mathbb P}\,$ are calculated under $\, \mathbb Q\,$ and $\, \mathbb P\,$, respectively, given the filtration $\, \{\mathcal F_{t}\}_{t\ge 0}\,$. 
Under the semimartingale stock price model \eqref{eq: price with shot noise2} with the mixed GFBM $\,Y\,$ in the regions (I) and (II)-1 of the parameter sets $\{ \, \gamma / 2 < \alpha < 1/2 + \gamma /2  \,, \, 0 < \gamma < 1 \,\}$, 
the pricing and hedging problems of various options (such as American and Asian options) are solved under the measure $\mathbb Q$ in the same way as in the standard Black-Scholes model. Thus, the Black-Scholes pricing formula is still valid under the long range dependence property of the driving mixed GFBM $Y$. 
{This semimartingale stock model generalizes the financial applications \cite{bender2011fractional} \cite{cheridito2001mixed} \cite{MR1912144} with the Hurst index $\,H \in (3/4, 1) \,$ to the case of $\,H \in (1/2, 1)\,$. }

\smallskip

Following \cite{cheridito2001mixed}, we next consider a slightly modified version $\,Y_{\epsilon}\,$ of the mixed FBM $Y$ and replace $\,Y\,$ by $\,Y_{\epsilon}\,$ in the stock price model \eqref{eq: price with shot noise2}: for a small constant $\,\epsilon >0\,$,  
\begin{equation} \label{eq: Yepsilon}
Y_{\epsilon}(t) :=\epsilon \widetilde{B}(t) + X(t) \,, \quad t \ge 0 \,.
\end{equation} 
One can check that the semimartingale property of $Y_\epsilon$ holds the same as $Y$. 
Indeed, since $\, Y_{\epsilon}(\cdot) :=\epsilon \widetilde{B}(\cdot) + X(\cdot) = \epsilon ( \widetilde{B}(\cdot) + (1/\epsilon) X(\cdot)) \, =\, \epsilon ( \widetilde{B} (\cdot) + X_{\epsilon}(\cdot)) \,$ with $\, (1/\epsilon)X(\cdot) =: X_{\epsilon}(\cdot)\,$, arguing that $\,Y_{\epsilon}(\cdot) \,$ is a semimartingale is very similar to the argument in Section \ref{sec-mix}. We can replace $\,X(\cdot) \,$ by $\,X_{\epsilon}(\cdot) \,$, $\,\Psi (u,v) \,$ in \eqref{eqn-Psi} by $\,\Psi_{\varepsilon} (u,v) := \epsilon^{-2} \Psi (u,v)\,$, $\,K(u,v) \,$ in \eqref{eq: lm4.1.1} by $\,\epsilon^{-2} K(u,v)\,$, $\,L(s, t) \,$ in \eqref{eqn-WH} by the solution $\, L_{\epsilon}(s,t)\,$ to the Wiener-Hopf integral equation 
\[
L_{\epsilon}(s,t) + \int^{t}_{0} L_{\epsilon}(r,t) K_{\epsilon} (r,s) {\mathrm d} r = - K_{\epsilon}(s,t) \, , \quad 0 \le s \le t \le T \, , 
\]
and $\, \ell (s,t) \,$ by the solution $\, \ell_{\epsilon} (s, t) \,$ to the Volterra equation 
\[
\ell_{\epsilon}(s, t)+ \int^{t}_{s} \ell_{\epsilon}(r,t) L_{\epsilon}(s,r) {\mathrm d} r = L_{\epsilon}(s, t) \, , \quad 0 \le s \le  T \, . 
\]
Then the rest of the arguments for $\,Y_{\epsilon}\,$ in \eqref{eq: Yepsilon} will follow in the same way as for $Y$. Thus, the pricing and hedging problems are solved in a similar way under the semimartingale stock model \eqref{eq: price with shot noise2} modified by replacing $\,Y\,$ by $\,Y_{\epsilon}\,$. 

We make the following observations on the process $\, Y_\epsilon\, $ in \eqref{eq: Yepsilon}. As $\epsilon\downarrow 0$, for the parameters $(\alpha,\gamma)$ in region (I) in Fig. \ref{fig-1}, the limit process $X=\lim_{\epsilon\downarrow0}Y_\epsilon$ remains to be a semimartingale, while in region (II)-1,  the limit process is no longer a semimartingale.   The limit in both regions allows an arbitrage (see section \ref{sec-arbitrage}
 below). 
 However, in both regions, the Hurst parameter $H$ for $X$ takes value in $(1/2,1)$. 
This is in contrast with the standard FBM mixture as studied in Cheridito \cite{cheridito2001mixed} where if $X=B^H$ for $H \in (3/4,1)$, then as $\epsilon\downarrow 0$, the limit remains to be a non-semimartingale.

\medskip

\subsection{Portfolio optimization with the mixed GFBM under logarithmic utility} \label{sec: 6.2.2}

Continuing from the previous section, we discuss the portfolio optimization under the stock price model \eqref{eq: price with shot noise2} with the mixed GFBM $\,Y\,$ in the regions (I) and (II)-1 of the parameter sets $\{ \, \gamma / 2 < \alpha < 1/2 + \gamma /2  \,, \, 0 < \gamma < 1 \,\}$ (see Fig. \ref{fig-1}(b)) with the logarithmic utility function $\,U(x) \, =\,  \log (x) \,$, $\, x > 0 \,$. For simplicity, we fix the instantaneous interest rate $\, r\,$ of the risk free asset and we set $\, Y_{0}(t) \, :=\, Y(t) \, /\,  \sqrt{2} \, =\,  (\widetilde{B}(t) + X(t))/\sqrt{2}  \,$, $\, \sigma_{0} \, :=\, \sqrt{2} \sigma \,$, so that the variance of $\, Y_{0}(\cdot) \,$ and the variance of the log stock price process $\, \log P(\cdot) \,$ are standardized \[
\, \text{Var} (Y_{0}(t)) \, =\,  \frac{1}{2}(t + t^{2H}) \, , \quad \, \text{Var} ( \log P(t)) = \sigma^{2}_{0} (t +t^{2H}) \,, \quad \, t \ge 0 \,. 
\]
Note that $\, \text{Var} ( X(t)) \, =\, t^{2H}\, $ and $\, \text{Var} ( \widetilde{B} (t)) \, =\,  t \,$, $\, t \ge 0 \,$. 

The investor invests the proportion $\,u^{0}_{t}\,$ into the risk free asset and the proportion $\,u^{1}_{t}\,$ into the stock with price $\, P(t)\,$ at time $\,t\,$. Here, we assume that $\,u^{1}_{\cdot}\,$ is an $\,\mathbb F^{Y}\,$-adapted, square integrable process with $\,u^{0}_{\cdot} + u^{1}_{\cdot} \, =\, 1 \,$. Then with the initial portfolio value $\,v_{0}\,$, the portfolio value process $\,V(t)\,$, $\, t \ge 0 \,$ satisfies the dynamic 
\begin{equation} \label{eq: portfolio value proc}
{\mathrm d} V(t) \, =\,  V(t) (u^{0}_{t} r + u^{1}_{t} \mu ) {\mathrm d} t + u^{1}_{t} V(t) \sigma_{0}  {\mathrm d} Y_{0}(t) \,, \quad  t \ge 0 \, .  
\end{equation}
The objective of the investor with the log utility function $\,U(\cdot) \,$ is to maximize the expected utility 
\begin{equation} \label{eq: log utility}
\mathbb E [ U(V(T)) ] \, =\,  \mathbb E [ \log (V(T) ) ] \, 
\end{equation}
at the end time $\, T\,$ of the investment horizon.  

\begin{proposition} \label{prop: 6} The solution to this maximization problem of the expected log utility \eqref{eq: log utility} of the portfolio value \eqref{eq: portfolio value proc} is given by the constant portfolio 
\begin{equation} \label{eq: const portfolio}
u^{0\ast}_{t} \, :=\, 1 - u^{1\ast}_{t} \, , \quad u^{1 \ast}_{t} \, :=\, \frac{\,\mu - r\,}{\,\sigma^{2}_{0}\,} \,, \quad t \ge 0 
\end{equation}
with the resulting optimal portfolio value process 
\begin{equation} \label{eq: opt port proc}
V^{\ast}(t) \, =\, v_{0} \exp \Big( \Big( r + \frac{\,\theta^{2}_{0}\,}{\,2\,}\Big)t + \theta_{0} Y_{0}(t) \Big) \,, \quad t \ge 0 \, 
\end{equation}
in the regions (I) and (II)-1 of the parameter sets $\{ \, \gamma / 2 < \alpha < 1/2 + \gamma /2  \,, \, 0 < \gamma < 1 \,\}$ and the line segment $\alpha = 0$, $\gamma \in [0, 1)$ (see Fig. \ref{fig-1}). Here, we use $\, Y_{0}(\cdot) \, =\,  Y(\cdot)/\sqrt{2}\,$, $\, \sigma_{0} \, =\,  \sqrt{2} \sigma \,$ and $\, \theta_{0} := (\mu - r) \, / \, \sigma_{0}\,$. 
\end{proposition}

\begin{proof} 
The proof follows from the semimartingale decomposition \eqref{eq: semi mart rep Y} of $\,Y(\cdot)\,$ in Proposition \ref{prop-mixture}. With $\,Y_{0} (\cdot)\,$ and $\, \sigma_{0} \,$, the portfolio value is given by 
\[
V(t) \, =\,  v_{0} \exp \Big(\int^{t}_{0} \Big[u_{s}^{0} r + u^{1}_{0} \mu - \frac{\,\sigma_{0}^{2}\,}{\,2\,}(u_{s}^{1})^{2} \Big] {\mathrm d} s + \int^{t}_{0} u_{s}^{1} \sigma_{0} {\mathrm d} Y_{0} (s) \Big) \, ,  
\] 
and hence, maximizing the expected log utility $\, \mathbb E [ \log (V(T)) ] \,$ is equivalent to maximizing the expectation 
\begin{equation} \label{eq: max ex log utility}
\begin{split}
& \mathbb E \Big[ 
\int^{T}_{0}\Big[u_{s}^{0} r + u^{1}_{0} \mu - \frac{\,\sigma_{0}^{2}\,}{\,2\,}(u_{s}^{1})^{2} \Big] {\mathrm d} s + \int^{T}_{0} u_{s}^{1} \sigma_{0} {\mathrm d} Y_{0} (s)
\Big]
\\
\, & =\,  \mathbb E \Big[ 
\int^{T}_{0}\Big[(u_{s}^{0} r + u^{1}_{0} \mu - \frac{\,\sigma_{0}^{2}\,}{\,2\,}(u_{s}^{1})^{2} \Big] {\mathrm d} s \Big] \, . 
\end{split}
\end{equation}
Here, when $\{ \, \gamma / 2 < \alpha < 1/2 + \gamma /2  \,, \, 0 < \gamma < 1 \,\}$, we have used the semimartingale decomposition \eqref{eq: semi mart rep Y} of $\, Y \,$, that is, 
\[
\int^{T}_{0} u_{s}^{1} \sigma_{0} {\mathrm d} Y_{0} (s) \, =\, \int^{T}_{0} u_{s}^{1} \sigma {\mathrm d} \Big ( \overline{W}(s) - \int^{s}_{0} \int^{\theta}_{0} L(r, \theta) {\mathrm d} \overline{W}(r) {\mathrm d} \theta \Big) \, , 
\]
where the square integrable function $\, L (r, \theta) \,$ is the solution to the Wiener-Hopf integral equation \eqref{eqn-WH} and $\, \overline{W}(\cdot)\,$ is another standard Brownian motion, and thus, the expectation $\, \mathbb E [ \int^{T}_{0} u_{s}^{1} \sigma_{0} {\mathrm d} Y_{0}(s) ] \, =\,  0 \,$. When $\, \alpha = 0 \,$, $\, \gamma \in [0, 1)\,$, $Y_{0} $ is a martingale, and hence, we have the expectation $\, \mathbb E [ \int^{T}_{0} u_{s}^{1} \sigma_{0} {\mathrm d} Y_{0}(s) ] \, =\,  0 \,$. 

The maximization of \eqref{eq: max ex log utility} with the constraint $\, u^{0}_{\cdot} + u^{1}_{\cdot} \, =\,  1\,$ is straightforward. The solution is given by \eqref{eq: const portfolio} and the resulting optimal portfolio value is given by \eqref{eq: opt port proc}. 
\end{proof}

\begin{remark} This resulting optimal constant portfolio \eqref{eq: const portfolio} from Proposition \ref{prop: 6} indicates the stability of the portfolio optimization under the price process model \eqref{eq: price with shot noise2} with the mixed fractional Brownian motion and log utility function for different parameter values in the regions (I) and (II)-1 (Fig. \ref{fig-1}(b)) with the Hurst index $\, H \in (1/2, 1)\,$. The expectation $\,\mathbb E [ V^{\ast}(t) ] \,$ and variance $\, \text{Var} (V^{\ast}(t) ) \,$ of the optimal portfolio $\,V^{\ast}(t)\,$ in \eqref{eq: opt port proc} depend on $\,H\,$: 
for $\, 0 \le t \le T\,$, 
\[
\mathbb E [ V^{\ast}(t) ] \, =\,  v_{0} \exp \Big( rt +\theta_{0}^{2} \Big( \frac{\,3\,}{\,4\,} t + \frac{\,1\,}{\,4\,} t^{2H}\Big) \Big) \, , 
\]
\[
\text{Var} ( V^{\ast}(t) ) \, =\,  v_{0}^{2} \exp 
\Big( rt + \theta_{0}^{2} \Big( \frac{\,3\,}{\,4\,} t + \frac{\,1\,}{\,4\,} t^{2H}\Big) \Big) \Big[ \exp 
\Big( rt + \theta_{0}^{2} \Big( \frac{\,3\,}{\,4\,} t + \frac{\,1\,}{\,4\,} t^{2H}\Big) \Big) - 1 \Big] \, . 
\]
The parameters $\,(\alpha , \gamma) \,$ appear in the covariances between different times, for example, 
\begin{equation}
\text{Cov} ( \log V^{\ast}(t), \log  V^{\ast}(s) ) \, =\,  \theta^{2} _{0}\text{Cov} (Y_{0}(t), Y_{0}(s) )  \, =\,  \frac{\,\theta^{2}_{0}\,}{2} ( (s \wedge t)  + \Psi (s, t)) \,  
\end{equation}
for $\, s, t \ge 0 \,$, where the covariance function $\, \Psi \,$ in \eqref{eqn-Psi} depends on $\, (\alpha, \gamma) \,$.

It is worth mentioning that this stability holds specifically under the log utility \eqref{eq: log utility} and the portfolio value process \eqref{eq: portfolio value proc} driven by  the mixed GFBM. 
See the recent work on the portfolio optimization under stochastic volatility models of Volterra-type for the power utility function \cite{MR4535336,MR4078803} and for more general utility function \cite{MR3974280}. The optimal portfolio choice under the general utility function and fractional stochastic environment is out of the scope of the current paper. 
\end{remark}

\medskip

\subsection{Comments on Arbitrage} \label{sec-arbitrage}

In the theory of asset pricing, the ``First Fundamental Asset Pricing Theorem" requires the existence of an equivalent martingale measure for no arbitrage, and works in the framework of semimartingales for pricing models. For stock price models with FBM $B^H$, since $B^H$ is a semimartingale if and only if $H=1/2$ \cite{liptser2012theory,rogers1997arbitrage}, arbitrage strategies have been discussed in both fractional Bachelier and Black-Scholes models \cite{rogers1997arbitrage,shiryaev1998arbitrage,salopek1998tolerance,cheridito2003arbitrage}. 
In particular, Rogers \cite{rogers1997arbitrage} constructed arbitrage for the fractional Bachelier model: a market with a money account $\xi_t =1$ with zero interest rate and a risky stock (no dividends or transaction costs) with price $\tilde{P}(t) = \tilde{P}(0) +\nu t+  \sigma B^H(t)$, for $H \in (0,1/2) \cup (1/2, 1)$ for $\, t \ge 0 \,$.


{For the GFBM $X$  in \eqref{def-X}, the parameter set that guarantees no-arbitrage is $\alpha=0$ and $\gamma\in [0,1)$ (resulting in $H\in (0,1/2]$), in which case $X$ is given in \eqref{eqn-X-alpha0} (either a standard Brownian motion when $\gamma=0$ or a time-changed Brownian motion when $\gamma \in (0,1)$).  } 
Recall Remark \ref{rem-sm}, $H=1/2$ corresponds to the line $\alpha=\gamma/2$ in Figure \ref{fig-1}, but if $\gamma\neq 0$, it is a ``fake" Brownian motion and it may lead to an arbitrage, because of the non-semimartingale property. 
Although we have shown that $\,X\,$ is a semimartingale with respect to the filtration $\mathcal{F}^B(\cdot)$ for $\alpha \in (1/2, 1/2+\gamma/2)$ and $\gamma \in (0,1)$ (resulting in $H \in (1/2,1)$ in region (I)), it is a process of finite variation. As it is shown in \cite{HPS84}, in a frictionless market with continuous trading, arbitrage opportunities are present without unbounded variation of the stock price process. Thus, the differentiable sample path in Proposition \ref{prop-sm-X} leads an arbitrage opportunity in this parameter range.

On the other hand, for the mixed BM and GFBM process $Y$ in \eqref{eq: Y}, because in the parameter range  $\alpha \in (1/2, 1/2+\gamma/2)$ and $\gamma \in (0,1)$ (resulting in $H \in (1/2,1)$ for $\, X\,$ in region (I)), the process $Y$ becomes a Brownian motion with a random drift of finite variation as given in  \eqref{eqn-wt-lambda},  there exists an equivalent martingale measure as given in Proposition \ref{prop-pricing-RN} (particularly, region (I) in Figure~\ref{fig-1}(b)), and hence, it forbids arbitrage. Moreover, by Proposition \ref{prop:LRD}, the GFBM in this region has LRD in the sense of Definition \ref{def:LRDSS}, hence the mixed GFBM process $Y$ has a component $X$ with LRD. In Section 5 of Rogers \cite{rogers1997arbitrage}, he constructed a Gaussian process with the same LRD as FBM and yet exhibiting as a semimartingale by choosing a different kernel function in the stochastic integral representation of FBM, that is, $(\epsilon + (t-s)^2)^{(2H-1)/4}$ instead of $(t-s)^{H-1/2}$. 
We provide another example as given in \eqref{eqn-wt-lambda}, which exhibits both LRD and semimartingale properties.

One can possibly construct  arbitrage using the GFBM similar as in \cite{shiryaev1998arbitrage}, but that involves the well definedness of stochastic integrals with respect to the GFBM, which we investigate in a future work.

\section*{Acknowledgements}
The authors are thankful to the editors and reviewers for their careful reading, several incisive suggestions and corrections.  
Tomoyuki Ichiba was supported in part by NSF grants DMS-1615229 and DMS-2008427. 
Guodong Pang was supported in part by CMMI-1635410,  DMS-1715875,  DMS-2216765 and Army Research Office through grant W911NF-17-1-0019. 
Murad S. Taqqu was supported in part by a Simons Foundation grant 569118 at Boston University.

\bigskip

\section{Appendix} \label{sec-app}

\subsection{Proof of Lemma \ref{lm:sq-integrability} } \label{sec-app-2}

\begin{proof}[Proof of Lemma \ref{lm:sq-integrability}] 

(i) 
Suppose that $\gamma/2 < \alpha <1/2+ \gamma /2 $ and $\, 0 <\gamma <1 \,$.
Using the inequality $\, (x+y)^{2}\le 2 (x^{2} + y^{2}) \,$ for $\, x, y > 0 \,$ and the symmetry of integral region, we obtain 
\begin{equation} \label{eq: lm4.1.2}
\int^{T}_{0} \int^{T}_{0} \bigg[ \frac{\partial^{2} \Psi}{\partial u \partial v}(u,v )\bigg]^{2} {\mathrm d} u {\mathrm d} v  \le 4 c^{4} \int^{T}_{0} \bigg(  \int^{v}_{0} (f_{1}(u,v))^{2} + (f_{2}(u,v))^{2} {\mathrm d} u \bigg) {\mathrm d} v \,. 
\end{equation}
Here, the first term $f_{1}(u,v) $ is bounded by 
\[
f_{1}(u,v) \, \le \int^{u}_{0} v^{\alpha-1} (u-\theta)^{\alpha - 1} \theta^{-\gamma} {\mathrm d} \theta \, =\,  \text{Beta} ( \alpha , 1 - \gamma)  \, u^{\alpha - \gamma}v^{\alpha -1} , 
\]
for $ u < v $. Since $ 0 < \alpha < 1$ and $\,(uw)^{\alpha -1 } (vw)^{\alpha -1} \ge ( 1 + uw)^{\alpha -1 }(1 + vw)^{\alpha -1}\,$ for $ u, v>0 $, we have a bound for  the second term $f_{2}(u,v) $, 
\begin{equation*}
\begin{split}
f_{2}(u,v) \, &=\,  \int^{\infty}_{0} (uv)^{\alpha-\gamma+1} (1 + uw)^{\alpha -1} (1 + vw )^{\alpha - 1} w^{-\gamma} {\mathrm d} w 
\\
&\le (uv)^{2\alpha - \gamma} \int^{\infty}_{1} w^{2\alpha -2 - \gamma} {\mathrm d} w  \\
&\quad + (uv)^{\alpha - \gamma +1} \int^{1}_{0} (1+uw)^{\alpha -1 } (1 + vw)^{\alpha -1 } w^{-\gamma} {\mathrm d} w \\
& \le \frac{(uv)^{2H-1}}{2 - 2 H } + \frac{\,(uv)^{\alpha - \gamma +1}\,}{\,1 - \gamma \,}\,. 
\end{split}
\end{equation*} 

Substituting these upper bounds of both terms and using again the inequality $\,(x+y)^{2} \le 2 (x^{2} + y^{2})\,$, $x, y > 0 $ for the second term, we obtain the estimates 
\begin{equation*}
\begin{split}
\int^{T}_{0} \int^{v}_{0} (f_{1}(u,v))^{2} {\mathrm d} u {\mathrm d} v \le &\,  [\text{Beta} ( \alpha , 1 - \gamma) ]^{2}\int^{T}_{0} \bigg( \int^{v}_{0} u^{2\alpha - 2 \gamma}v^{2\alpha - 2} {\mathrm d} u \bigg){\mathrm d} v  
\\
\, =  & \,  \frac{\,[\text{Beta} ( \alpha , 1 - \gamma) ]^{2} \, T^{4 \alpha - 2 \gamma}\,}{\,2(2\alpha -  \gamma) (2\alpha - 2 \gamma +1)\,} < \infty \, , 
\end{split}
\end{equation*} 
and
\begin{equation*}
\begin{split}
&  \int^{T}_{0} \int^{v}_{0} (f_{2}(u,v))^{2} {\mathrm d} u {\mathrm d} v \\
\le &  \int^{T}_{0} \bigg[ \int^{v}_{0} \bigg( \frac{(uv)^{2H-1} }{2 -2H } + \frac{\,(uv)^{\alpha - \gamma +1}\,}{\,1 - \gamma \,} \bigg)^{2} {\mathrm d} u \bigg]  {\mathrm d} v
\\
 \le&\,   2\,  \int^{T}_{0}  \int^{v}_{0} \bigg( \frac{(uv)^{4 H-2} }{(2-2H)^{2} } + \frac{\,(uv)^{2\alpha - 2\gamma +2}\,}{\,(1 - \gamma )^{2} \,} \bigg) {\mathrm d} u   {\mathrm d} v 
\\
 \, =& \, \frac{\,T^{8H+1}\,}{\,(2 - 2H)^{2} (4H-1)(4H+1)\,} + \frac{\,T^{4H-2 \gamma +4}\,}{\,(1-\gamma)^{2} (4 H - 2 \gamma +3) (2 H- \gamma + 2)\,} \\
 \, <& \, \infty \, . 
\end{split}
\end{equation*} 
The right hand sides are finite when $\, 2 \alpha > \gamma  \,$ and $\, 0 < \gamma < 1\,$. 

Therefore, combining these estimates with \eqref{eq: lm4.1.2}, we conclude the second derivative $\, K(u,v) \,$ is square integrable in $\,(0, T) \times (0, T) \,$.

\noindent (ii)  Suppose that $\, -1/2+ \gamma/2< \alpha \le \gamma/2 \,$ and $\, 0 < \gamma < 1\,$. 
 Since $\, (x+y)^{2} \ge x^{2} \,$ for $\,x, y > 0 \,$, we shall show 
\[
\int^{T}_{0} \int^{T}_{0} [K(u,v)]^{2} {\mathrm d} u {\mathrm d} v \, \ge \, 2 \int^{T}_{0} \int^{v}_{0} [f_{1}(u,v)]^{2} {\mathrm d} u {\mathrm d} v \, =\, \infty \, . 
\]
To do so, by the change-of-variable and by Jensen's inequality, we observe that if $\, u \le v \,$,
\begin{equation*}
\begin{split}
f_{1}(u,v) \, =\, & u^{\alpha - \gamma }  \int^{1}_{0} (v - u w)^{\alpha -1}w^{-\gamma} ( 1- w)^{\alpha -1} {\mathrm d} w \\
 \, =\, & u^{\alpha - \gamma } \text{Beta} (\alpha , 1 - \gamma)  \int^{1}_{0} (v - u w)^{\alpha -1}\frac{ w^{-\gamma} ( 1- w)^{\alpha -1}}{ \text{Beta} (\alpha , 1 - \gamma)} {\mathrm d} w\\
 \ge \, &  \text{Beta} (\alpha , 1 - \gamma)\,  u^{\alpha - \gamma}\, \bigg ( v - u \cdot \frac{\,\alpha \,}{\,\alpha + 1 - \gamma \,} \bigg)^{\alpha - 1} \, , 
\end{split}
\end{equation*}
because $\, w \mapsto (v - u w)^{\alpha - 1}\,$, $\, 0 < w < 1\,$ is a convex function and the expectation of Beta distribution with parameters $\,(\alpha, 1 - \gamma) \,$ is $\, \alpha / ( \alpha + 1 - \gamma ) \,$. Thus, we have a lower bound for $\, \int^{v}_{0} [f_{1}(u,v)]^{2} {\mathrm d} u \,$, that is, 
\begin{equation*}
\begin{split}
& \int^{v}_{0} [f_{1}(u,v)]^{2} {\mathrm d} u \\
 \, \ge  \, & \int^{v}_{0}[ \text{Beta} (\alpha , 1 - \gamma)]^{2}\,  u^{2 (\alpha - \gamma)}\, \bigg ( v - u \cdot \frac{\,\alpha \,}{\,\alpha + 1 - \gamma \,} \bigg)^{2 (\alpha - 1)} {\mathrm d} u \\
= \, & [ \text{Beta} (\alpha , 1 - \gamma)]^{2} v^{4 \alpha - 2 \gamma -1} \int^{1}_{0} \bigg( 1 - \theta \cdot \frac{\,\alpha \,}{\,\alpha + 1 - \gamma\,} \bigg)^{2 \alpha - 2 } \theta ^{2\alpha - 2 \gamma} {\mathrm d} \theta\\
\, = \, & [ \text{Beta} (\alpha , 1 - \gamma)]^{2} \cdot \bigg( \frac{\,1 - \gamma \,}{\, \alpha  + 1 - \gamma \,}\bigg)^{2\alpha - 2} \frac{ v^{4 \alpha - 2 \gamma -1}  }{2 \alpha - 2\gamma + 1 } \, 
\end{split}
\end{equation*}
for $\, 0 < v < T\,$. However, if $\, 2 \alpha \le \gamma \,$, this lower bound is not integrable over $\, (0, T ) \,$, and hence, $\, K(\cdot, \cdot) \,$  is not square integrable over $\, (0, T) \times (0, T) \,$.
\end{proof}

\subsection{Proof of Proposition \ref{prop:LRD}}\label{app:LRD}
Recall the  covariance function $\, C_{\eta}(t)$ in \eqref{eqn-Ceta} for the Lamperti transform $\, \eta(t) \, :=\, e^{-Ht} X (e^{t})\,$ for $t\in \mathbb R$. We have 
\begin{equation} \label{eq: Cetat}
\begin{split}
C_{\eta}(t) \, =\, c^{2} e^{-Ht} \bigg[ & \int^{1}_{0} (1 - s)^{\alpha} (e^{t} - s)^{\alpha} s^{-\gamma}  {\mathrm d} s \\
& + \int^{\infty} _{0} [ (1+s)^{\alpha} - s^{\alpha}] [(e^{t} + s)^{\alpha} - s^{\alpha} ] s^{-\gamma} {\mathrm d} s	  \bigg] \,, \quad t > 0 \, . 
\end{split}
\end{equation}

To show that $\, \alpha > 0 \,$ implies the inequality \eqref{eq: LRDSS}, it suffices to show that 
\begin{equation} \label{eq: 64}
\lim_{t\to \infty} \frac{\,1\,}{\,t\,} \log \int^{1}_{0} (1 - s)^{\alpha} (e^{t} - s)^{\alpha} s^{-\gamma} {\mathrm d} s \, \ge \,  \frac{\,\alpha \,}{\,2\,} > 0 \, , 
\end{equation}
since the second integral term in \eqref{eq: Cetat} is positive. Using the inequality $\, (e^{t} - s)^{\alpha} \ge (e^{t} -1 )^{\alpha}  \ge e^{\alpha t / 2}\,$ for $\, 0 \le s \le 1 \,$, $\, \alpha > 0 \,$, $\, t > 0 \,$, we obtain 
\[
\int^{1}_{0} (1-s)^{\alpha} ( e^{t} - s)^{\alpha} s^{-\gamma} {\mathrm d} s \ge   \int^{1}_{0} ( 1-s)^{\alpha} s^{-\gamma} {\mathrm d} s \cdot e^{\alpha t/2}  \, . 
\]
Thus, taking the logarithm first, dividing by $\, t > 0 \,$, and then taking the limit of both sides as $\,t \to \infty\,$, we obtain the desired result \eqref{eq: 64}. Hence, if $\, \alpha > 0 \,$, then $\, X\,$ in \eqref{def-X} is long range dependence in the sense of Definition \ref{def:LRDSS}. 

For necessity, assume $\, -1/2 < \alpha \le 0 \,$. From the first integral term in \eqref{eq: Cetat}, we have 
\begin{equation} \label{eq:66}
\begin{split}
& \lim_{t\to \infty}  \frac{\,1\,}{\,t\,} \log \int^{1}_{0} (1 - s)^{\alpha} (e^{t} - s)^{\alpha} s^{-\gamma} {\mathrm d} s 
\\
& \le \lim_{t\to \infty} \frac{\,1\,}{\,t\,} \log \Big((e^{t} - 1)^{\alpha} \cdot \int^{1}_{0} (1 -s)^{\alpha} s^{-\gamma} {\mathrm d} s  \Big)\, =\,  \alpha \le 0 \, . 
\end{split}
\end{equation}
For the second integral term in \eqref{eq: Cetat}, we observe  that if $\, \alpha \le 0 \,$, $\, t > 0 \,$, 
\begin{equation}
\begin{split}
& \int^{1}_{0} [ (1+s)^{\alpha} - s^{\alpha}] [(e^{t} + s)  ^{\alpha} - s^{\alpha} ] s^{-\gamma} {\mathrm d} s\\
& = \int^{1}_{0} [ s^{\alpha} - (1+s)^{\alpha} ] [s^{\alpha} - (e^{t} + s)  ^{\alpha} ] s^{-\gamma} {\mathrm d} s \\
& \le \int^{1}_{0} [ s^{\alpha} - (1+s)^{\alpha} ] s^{\alpha-\gamma}   {\mathrm d} s \, \le \, \frac{\,1\,}{\,2 \alpha - \gamma + 1\,} - \frac{\,2^{\alpha}\,}{\, \alpha - \gamma + 1\,} \, < \infty \, 
\end{split}
\end{equation}
and also, defining $\, \widetilde{\alpha} \, :=\, -\alpha > 0 \,$, we obtain that 
\begin{equation} \label{eq:68}
\begin{split}
& \int^{\infty}_{1} [ (1+s)^{\alpha} - s^{\alpha}] [(e^{t} + s)  ^{\alpha} - s^{\alpha} ] s^{-\gamma} {\mathrm d} s\\
=& \int^{\infty}_{1} [ s^{\alpha} - (1+s)^{\alpha} ] [s^{\alpha} - (e^{t} + s)  ^{\alpha} ] s^{-\gamma} {\mathrm d} s \\
= & \int^{\infty}_{1} \Big( 1 - \Big( \frac{\,s\,}{\,1+s\,} \Big)^{ \widetilde{\alpha}} \Big) \Big( 1 - \Big( \frac{\,s\,}{\,e^{t}+s\,} \Big)^{ \widetilde{\alpha}} \Big) s^{2 \alpha -\gamma} {\mathrm d} s \\
\le & \int^{\infty}_{1} s^{2\alpha -\gamma} {\mathrm d} s \, =\,  \frac{\,1\,}{\,2 \alpha - \gamma + 1\,} < \infty \, .  
\end{split}	
\end{equation}
Combining \eqref{eq: Cetat} with \eqref{eq:66}--\eqref{eq:68}, we conclude that $\, \lim_{t\to \infty} (1/t)  \log C_{\eta}(t) + H \, \le 0 \, $ when $\, -1/2 < \alpha \le 0 \,$. Finally, if $\, \gamma \, =\,  0 \,$, it is the case of the FBM $\, B^{H}\,$, and the covariance function $\, C_{\eta} (\cdot) \,$ of $\, \eta(t) := e^{- Ht} B^{H}(e^{t})\,$, $\, t \in \mathbb R\,$ is given by \eqref{eq: CovLampertiFBM} (see \cite{MR1781002} for the related computations). Then by direct calculations, we obtain \eqref{eq: CovLampertiFBM2}. Indeed, the limit for the FBM in \eqref{eq: CovLampertiFBM2} is rewritten as  
\begin{equation}
\lim_{t\to \infty} \frac{1}{\, t \, } \log ( e^{Ht} + e^{-Ht} - (e^{t/2} - e^{-t/2})^{2H}) + H \, . 
\end{equation}
Applying L'H\^opital's rule to this limit, we reduce it to 
\begin{equation} \label{eq: 71}
\begin{split}
& \lim_{t\to \infty} H \cdot \frac{e^{Ht} - e^{-Ht} - (e^{t/2} - e^{-t/2})^{2H-1} (e^{t/2} + e^{-t/2})}{e^{Ht} + e^{-Ht} - (e^{t/2} - e^{-t/2})^{2H}} + H 
\\
\, & =\,  \lim_{t\to \infty} H \cdot \frac{1 - e^{-2Ht} - (1-e^{-t})^{2H-1} (1 + e^{-t})}{1 + e^{-2Ht} - (1 - e^{-t})^{2H}} + H 
\\
\, & =\,  \lim_{\varepsilon \to 0} H \cdot \frac{1 - \varepsilon^{2H} - (1-\varepsilon)^{2H-1} (1 + \varepsilon) }{1 + \varepsilon^{2H} - (1  -  \varepsilon)^{2H}} + H \, =\,  2 H - 1 \, , 
\end{split}
\end{equation}
where we changed the variable with $\, \varepsilon := e^{-t}\,$ and applied L'H\^opital's rule in the last equality.  

This concludes the proof of Proposition \ref{prop:LRD}.  \hfill$\Box$

\bigskip

\bibliographystyle{abbrv}
\bibliography{SM-GFBM-FS}

\end{document}